\documentclass[twoside,leqno,twocolumn]{article}

\usepackage[letterpaper]{geometry}

\usepackage{ltexpprt}
\usepackage{mathbbol}

\usepackage{amsmath}
\usepackage{url}

\usepackage{multirow}

\usepackage{graphicx}

\usepackage{enumitem}

\usepackage{algorithm}
\usepackage{algpseudocode}

\algnewcommand{\LeftComment}[1]{\Statex \(\triangleright\) #1}

\newcommand{\R}{\mathbb{R}}
\newcommand{\Z}{\mathbb{Z}}
\newcommand{\gr}{\mathrm{gr}}

\newcommand{\ignore}[1]{}

\newcommand{\Ex}{\mathit{Ex}}

\usepackage{hyperref}

\begin{document}

\title{
Fast Minimal Presentations of Bi-graded Persistence Modules
\thanks{Supported by Austrian Science Fund (FWF) grant number P 29984-N35.}
}
\author{Michael Kerber\thanks{Graz University of Technology, Graz, Austria}
\and Alexander Rolle\footnotemark[2]}
%\author{Anonymous submission}
\date{}

\maketitle

% Default Copyright Statement
%\fancyfoot[R]{\scriptsize{Copyright \textcopyright\ 20XX by SIAM\\
%Unauthorized reproduction of this article is prohibited}}

% Depending on which copyright you agree to when you sign the copyright form, the copyright 
% can be changed to one of the following after commenting out the default copyright statement
% above.

%\fancyfoot[R]{\scriptsize{Copyright \textcopyright\ 20XX\\
%Copyright for this paper is retained by authors}}

%\fancyfoot[R]{\scriptsize{Copyright \textcopyright\ 20XX\\
%Copyright retained by principal author's organization}}
%\pagenumbering{arabic}
%\setcounter{page}{1}%Leave this line commented out.

\begin{abstract} %\small\baselineskip=9pt 
Multi-parameter persistent homology is a recent branch of topological data
analysis. In this area, data sets are investigated through the lens of homology
with respect to two or more scale parameters. 
The high computational cost of many algorithms calls for a preprocessing step
to reduce the input size. 
%A minimal presentation is the smallest representation of a persistence module with equivalent algebraic information. TEST2.
In general, a minimal presentation is the smallest possible representation of a persistence module.
Lesnick and Wright~\cite{lw-computing}
proposed recently an algorithm (the LW-algorithm) for computing minimal presentations based on matrix reduction. 

In this work, we propose, implement and benchmark several improvements
over the LW-algorithm. Most notably, 
we propose the use of priority queues
to avoid extensive scanning of the matrix columns, which constitutes the 
computational bottleneck in the LW-algorithm, and
we combine their algorithm 
with ideas from the multi-parameter chunk algorithm 
by Fugacci and Kerber~\cite{fk-chunk}.
Our extensive experiments show that our algorithm outperforms the LW-algorithm
and computes the minimal presentation for data sets with millions of simplices
within a few seconds.
Our software is publicly available.
%Our code will be made publicly available on publication of this manuscript.
\end{abstract}

%\clearpage

\section{Introduction.}
\label{sec:introduction}
\paragraph{Motivation.}
Persistent homology is a multi-scale approach to extract information from 
a data set that would remain hidden using conventional methods.
The idea is to consider the data at multiple scales and track the topological
evolution as the scale changes. The success of this theory, 
%coining the term \emph{topological data analysis} 
which is a primary tool of \emph{topological data analysis}, 
has been manifested in countless
applications, e.g.~\cite{rbd-decoding,bmmps-persistent,cosmo2,phr-multiscale}.
One major reason for this success is the existence of fast algorithms to compute
the required topological invariants.

Multi-parameter persistent homology extends the above concepts 
%in a way that the data is not only filtered by one, but an arbitrary number of (independent)
%scale parameters
by filtering the data by several independent scale parameters, 
rather than just one. 
Indeed, it is quite often natural 
to look at multi-dimensional parameter spaces in data analysis, 
and this theory allows for a more fine-grained topological analysis. 
As a simple example, when studying point cloud data, it is common to
``thicken'' the point set (i.e., replacing points with balls of radius $r>0$)
in the single-parameter setup. Since the results of 
such an analysis are unstable with respect to outliers, 
the point cloud is usually preprocessed by removing outliers first. 
This preprocessing introduces another scale parameter
that determines how aggressively points are classified as outliers
(see, e.g.,~\cite{bcdfow-topological}).

A drawback of multi-parameter persistent homology is that one has to deal
with more complicated problems, already for the case of two parameters
that we consider in this work: 
a complete discrete topological invariant does not exist~\cite{cz-theory},
making the algebraic objects richer (and more difficult) than in the
single-parameter setting. This also extends to computational problems:
for instance, computing the universal distance (i.e.,``most meaningful''~-- see \cite{lesnick-theory} for details) is NP-hard for two parameters,
but polynomial time (and practically efficient) for one parameter~\cite{kmn-geometry}.
For several other problems, polynomial-time algorithms have been described
recently, such as for computing the matching distance~\cite{bcfg-new,klo-exact,kn-efficient}, a multi-parameter kernel~\cite{cfklw-kernel}, 
and the decomposition of multi-parameter modules into indecomposable elements~\cite{dx-generalized} as well as weaker decompositions~\cite{ckmw-elder}.
Despite being polynomial time, the complexity bounds 
typically have rather large exponents and this is also reflected in
poor behavior for large input sizes in practice.

To improve all aforementioned algorithms, we want to preprocess the input, 
%such that its size reduces significantly, but it still yields the same
%topological information. 
reducing its size as much as possible, while retaining all the relevant homological information. 
Such a reduction is possible because usually,
the input is given as a bi-graded simplicial complex,
%(or, more generally, a bi-graded cell complex), 
but for the sake
of persistent homology, 
%one is only interested in the homology groups
%at each scale which are a rather coarse summary of the data (in fact,
%for a fixed homology-dimension, the group is characeterized by a single integer
%denoting the rank of the group). 
one is only interested in the evolution of the homology of the complex 
across the various bi-grades.
As homology is only a coarse topological summary of a complex, 
we can hope for a more succinct representation.

A \emph{minimal presentation} is such a succinct
representation.
In short, a presentation encodes the homology by a set of generators
and relations, each of which is associated with a bi-grade. 
The generators create homological features at the corresponding bi-grade,
and the relations turn a certain linear combination of generators trivial
at their bi-grade. This data can be encoded in a matrix where rows represent
generators and columns represent relations, and each row and column also
stores a bi-grade. This matrix is the main object of our studies.
A presentation is called minimal if among all presentations of the same
homological data, it possesses the minimal number of generators and relations
(we refer to Section~\ref{sec:definitions} for a precise definition).

A recent algorithm for computing minimal presentations
by Lesnick and Wright~\cite{lw-computing}
(the LW-algorithm, from now on)
is a key component of RIVET~\cite{rivet},
the most prominent software package supporting $2$-parameter persistent homology calculations.
RIVET allows for an interactive exploration of bi-filtered data sets
and has received attention in various data analysis applications,
including neuroscience~\cite{kanari_et_al, sizemore_et_al}, physics~\cite{cole_shiu},
and machine learning~\cite{vipond}.

\paragraph{Our contribution.}
We describe a fast algorithm to compute a minimal presentation.
%More precisely, we are introducing improvements
%for the recent algorithm by Lesnick and Wright~\cite{lw-computing}
%(called LW-algorithm from now on)
%drastically improving its performance.
More precisely, we introduce improvements to the LW-algorithm,
drastically improving its performance.

The LW-algorithm
arrives at a minimal presentation by a sequence of column operations,
employing a smart traversal strategy of the columns 
(we describe the approach in detail in Section~\ref{sec:lesnick-wright}).
However, it turns out that the bulk of its running time is not spent
on performing the actual work (i.e., the column operations), but \emph{searching}
for the work to be done. 
%Our improvements help the algorithm
%to find the required operations much more efficiently.

We introduce several major improvements
which are all variations of the same theme:
at several occasions,
%the algorithms scans a range of columns to update
%columns. 
the algorithm scans a range of columns in order to find which need to be updated. 
However, most of the scanned columns typically remain unchanged.
%and the bulk of the runtime is spent on scanning columns that are skipped.
Our observation is that those columns that need an update can be 
predicted by earlier steps of the algorithm, hence we can schedule them
for updates (using a priority queue) and avoid the scan completely.
This simple observation removes the major computational bottlenecks
and leads to immense performance gains.

We obtain further speed-ups by using the multi-chunk algorithm of Fugacci
and Kerber~\cite{fk-chunk} as a preprocessing step: this algorithm transforms
a bi-graded cell complex into a smaller bi-graded cell complex 
with the same persistent homology. 
The result is not minimal in terms of presentations, but this algorithm
typically gets rid of many cells that the LW-algorithm would only remove in the last step.
Hence, preprocessing results in starting the LW-algorithm
with fewer rows and columns. As the chunk algorithm only needs
a fraction of the runtime, the preprocessing is efficient in practice.
We investigated further variants:
\begin{itemize}
\item We implemented a clearing optimization (that is already hinted at
in~\cite{lw-computing}) to avoid certain column operations in the algorithm.
Unlike in the single-parameter case where clearing is an important tool
for efficiency~\cite{bkr-clear}, it only yields a small improvement in our scenario.
Part of the reason is that the chunk preprocessing already removes
most of the potential for clearing, and it does so more efficiently.

\item %Including chunk-preprocessing,
The algorithm consists of $5$ main sub-steps, out of which $3$ can be
easily parallelized, and the other two are independent and can be ran
parallel to each other as well. We show that parallelization yields
to additional speed-ups.

\item We use a matrix data structure derived 
from the \textsc{phat} library. As observed in~\cite{bkrw-phat}, the column
type of the matrix has a significant influence on the performance.
Our implementation is generic in the column type, so that we could
benchmark our approach with all columns types that \textsc{phat} offers.
%Perhaps surprisingly, best results were obtained using sorted arrays
%as column types (option ``vector-vector'' in \textsc{phat}),
%in contrast to the single-parameter case where that representation
%performs rather poorly.
\end{itemize}

We tested our C++-implementation on a variety of data sets,
including triangular mesh data, bifiltered flag complexes, and others.
The speed-ups obtained by our improvements are significant for all tested instances,
and in many cases, the runtime and memory consumption changes from a quadratic
to a near-linear empirical behavior;
see Sections~\ref{sec:improvements} 
and~\ref{sec:further_experiments}.
In particular, computing the minimal presentation for complexes
with millions of cells is now possible within a few seconds.
%Our software is ready for publication; we provide a preliminary
%anonymous repository\footnote{\url{https://www.amazon.de/clouddrive/share/QhM3hxgyMxrtjeqdr1P7rLVCUUpfjCAkRjEdGWWfjbG}}
%for the reviewing process.
Our software is called mpfree, and it is available.\footnote{\url{https://bitbucket.org/mkerber/mpfree}}

\paragraph{Further related work.}
The LW-algorithm is specialized for the case of bi-graded persistence modules.
However, minimal presentations can be computed for more general (graded) modules,
and more general algorithms are provided by computer algebra systems
like Macaulay or Singular. Lesnick and Wright~\cite{lw-computing} compare
their implementation with these systems and show that their algorithm
is much faster in this special case.

There are alternatives to minimal presentations 
to reduce the size of cell complexes. Besides the aforementioned result
by Fugacci and Kerber~\cite{fk-chunk}, a line of research simplifies
cell complexes through collapses based on Discrete Morse Theory~\cite{akl-reducing,aklm-acyclic,sifl-computing}. 
These approaches empirically yield
a similar compression rate as the chunk reduction, 
but only provide weaker
theoretical guarantees in terms of minimality, 
and their current implementation
is not as fast as the chunk algorithm.

This work is similar in spirit to several papers improving the performance
of the algorithm to compute persistence diagrams in a single parameter.
See~\cite{optgh-roadmap} for a survey. Most algorithmic improvements
are based on rather simple observations which nonetheless have to be combined
in the right way to yield an efficient result~-- the \textsc{Phat} library
is dedicated to studying the different algorithmic variants;
see~\cite{bkrw-phat} for a discussion how different parameters 
can change the outcome drastically.

\paragraph{Scope.}
We focus in this paper on the algorithm engineering aspect of our work.
That means that we introduce algebraic concepts only as much as needed
to describe the algorithmic steps and our improvements, 
and shift the correctness proofs to the appendix.
%A more comprehensive description with more explanations of the underlying
%algebra is planned for the full version of this article.

%%%%%%%%%%%%%%%%%%%%%%%%%%%%%%%%%%%%%%%%%%%%%%%%%%%%%%%%%%%%%%%%%%%%%%%%%%%

\section{Preliminaries.}
\label{sec:definitions}
\paragraph{Bi-Graded matrices.} 
A \emph{bi-grade} is an element of $\Z^2$. 
A bi-grade $g=(g_x,g_y)$ is \emph{smaller} than $h=(h_x,h_y)$,
written $g\leq h$, if $g_x\leq h_x$ and $g_y\leq h_y$.
%$\leq$ is only a partial order on the grades, and this fact is the ultimate
%reason of the difficulty of multi-parameter persistent homology
%compared to the single-parameter variant.

We fix the base field $K:=\Z_2$ for simplicity, although the algorithm
can be rephrased for any field $K$ with little effort. 
A \emph{bi-graded matrix} is
an $(m\times n)$-matrix over $K$, where each row and each column
is annotated with a bi-grade
such that whenever the $(i,j)$-entry of the
matrix is $1$, then $g_{(i)}\leq g^{(j)}$, where $g_{(i)}$ is the grade
of the $i$-th row and $g^{(j)}$ is the grade of the $j$-th column.
In what follows, we will skip the ``bi''-prefix and only talk about 
\emph{graded matrices} and \emph{grades}.

We will assume that a graded matrix is stored in some sparse column 
representation, that is, the row indices of non-zero entries in a column
are stored in some container data structure (e.g., a dynamic vector
or a linked list). Also, we assume that
a graded matrix always stores its rows and columns
in co-lexicographic order with respect to the grades, with rows/columns
of the same grade ordered arbitrarily. Note that this order
is a refinement of $\leq$ to a total order. 
%Such an order is easy to 
%obtain by sorting rows and columns and re-indexing the column data.

For a non-zero column, the \emph{pivot} is the largest row index with 
non-zero entry with respect to the fixed total colex-order of rows.
We call a column \emph{local} if the grade of the pivot row equals
the grade of the column.

\paragraph{Persistence modules.}
A \emph{persistence module} is a family $(V_p)_{p\in\Z^2}$
of $K$-vector spaces for each grade,
together with maps $f_{p\rightarrow q}:V_p\rightarrow V_q$
for every pair of grades with $p\leq q$ which are \emph{functorial}, 
that means, $f_{p\rightarrow p}$ is the identity and $f_{q\rightarrow r}\circ f_{p\rightarrow q}=f_{p\rightarrow r}$ for $p\leq q\leq r$.
The term ``module'' comes from the fact that this object carries the
algebraic structure of a $\Z^2$-graded module over the polynomial ring $K[x,y]$~\cite{cz-theory,ck-representation}.

%Graded matrices induce persistence modules in two different ways. 
We can define persistence modules using graded matrices.
Fixing a grade $p\in\Z^2$, let $M_{\leq p}$ denote the submatrix of $M$
that contains the rows and columns with grades $\leq p$. 
Note that a graded $m\times n$ matrix $M$ induces a linear map $K^n\to K^m$
by matrix-vector multiplication, and the same is true for $M_{\leq p}$,
adjusting the dimensions of domain and co-domain to the dimension of the submatrix.
We call a pair $(A,B)$ with $A$ an $\ell\times n$ and $B$ an $n\times m$-matrix a \emph{free implicit representation (firep)} 
if the row grades of $A$ coincide with the column grades of $B$,
and $BA=0$. In that case, for every $p$, also $B_{\leq p}A_{\leq p}=0$ holds,
implying that the image of $A_{\leq p}$ is contained 
in the kernel of $B_{\leq p}$. Hence we can define
\[
V_p:=(\mathrm{ker} B_{\leq p}) / (\mathrm{im} A_{\leq p})
\]
and $f_{p\to q}$ to be map that assigns to an element of $\mathrm{ker} B_{\leq p}$ an element of $\mathrm{ker} B_{\leq q}$ by padding additional dimensions
with zeros. Note that this construction, despite its algebraic formalism,
is very natural in practice: Given a bi-graded simplicial complex $C$
and a dimension $d$, picking for $A$ the $((d+1),d)$-boundary matrix
and for $B$ the $(d,(d-1))$-boundary matrix yields a free implicit representation.
For readers unfamiliar with these concepts, we refer to Appendix~\ref{app:simplicial} for more explanations.

A single graded $(m\times n)$-matrix $M$ can be interpreted
as a persistence module: the persistence module defined by the firep $(M, 0)$.
In detail:
denoting $B:=\{b_1,\ldots,b_m\}$ as basis elements of $K^m$, we grade $b_i$ with the $i$-th row grade of $M$.
We interpret every column of $M$ as a relation, 
a linear combination $r_i$
of elements in $B$, yielding $R=\{r_1,\ldots,r_n\}$. Then, we set
\[
V_p:=\langle B_{\leq p}\rangle / \langle R_{\leq p}\rangle
\]
where $\langle B_{\leq p}\rangle$ is the span of all basis vectors that have grade $\leq p$
and, likewise, $\langle R_{\leq p}\rangle$ is the span of all relations at grade $\leq p$.
The linear maps $V_{p\to q}$ are induced by the natural inclusions
$B_{\leq p}\to B_{\leq q}$. This interpretation of a graded matrix
as persistence module is called a \emph{presentation} of a graded module,
where the set $B$ is called the \emph{generators} and set $R$ is called
the \emph{relations} of the presentation. A presentation is \emph{minimal}
if among all possible presentations of the persistence module $(V_p)$,
it has the smallest number of generators and relations.
We provide more algebraic background on this at the beginning of Appendix~\ref{app:min_pres_algebra}.

\paragraph{Experimental setup.}
We will interleave the algorithmic description and the experimental evaluation
to show the impact of our improvements.
Our implementation is written in C++, compiled with gcc-7.5.0,
and ran on a workstation with an Intel(R) Xeon(R) CPU E5-1650 v3 CPU (6 cores,
3.5GHz) and 64 GB RAM, running GNU/Linux (Ubuntu 16.04.5).
We measured the overall time and memory consumption using the linux
command \texttt{/usr/bin/time}, and the \texttt{timer} library of the
\textsc{Boost} library to measure the time for subroutines.
%Our source code, the input data, and the benchmark scripts are 
%available (see the footnote in the introduction).
The published version of our software has benefited from low-level improvements
since we ran these experiments, but the performance is not significantly different. 
The input data and the benchmark scripts are 
available on request. 

In Sections~\ref{sec:lesnick-wright} and \ref{sec:improvements},
we will use running examples 
generated by the convex hull of $n$ points 
on a sphere sampled uniformly at random.
By Euler's formula, the convex hull consists of exactly
$e=3n-6$ edges and $f=2n-4$ triangles.
This yields a free implicit representation $(A,B)$,
where $A$ is the $(e\times f)$-boundary matrix for triangles and edges,
and $B$ the $(n\times e)$-boundary matrix for edges and vertices.
The grades of a vertex are simply its $x$- and $y$-coordinates
(ignoring the $z$-coordinate), and the $x$-/$y$-grades
of an edge and a triangle is just the maximal $x$-/$y$-coordinate
among its boundary vertices
(the fact that the grades are in $\R^2$ rather than $\Z^2$
does not make a difference because there are only finitely many of them and so
they can be mapped bijectively to $\Z^2$ preserving orders).
This is also known as the \emph{lower star (bi-)filtration}
which generalizes a commonly used filtration type
for a single parameter
and has been used as a benchmark example in previous multi-parameter work~%
\cite{sifl-computing,fk-chunk}.
%We choose this example because it
%demonstrates the effects of our improvement steps most clearly;
%however, 
We also performed
experiments on various other types of data and our speed-ups significantly improve
all tested instances; 
we report on this in Section~\ref{sec:further_experiments}.
For all randomly generated data sets, the listed results are averaged
over $5$ random instances of the same size.

\section{The Lesnick-Wright Algorithm.}
\label{sec:lesnick-wright}
We rephrase the algorithm from~\cite{lw-computing} in combinatorial terms.
While doing that, we try to give some intuition about why certain steps
are performed in the algorithm; 
%however, arguing formally about correctness
%requires the introduction of multiple concepts from algebra, and 
we refer to the original paper for correctness proofs.

\paragraph{Overview.}
The input of the LW-algorithm is a free implicit representation $(A,B)$,
and the output is a minimal presentation matrix $M$.

The algorithm consists of $4$ steps, where the first $3$ steps convert
the firep $(A,B)$ into a presentation matrix $M'$ 
(called \emph{semi-minimal} in~\cite{lw-computing}), 
and the last step minimizes $M'$ into $M$.
We describe the four steps on a high level first:

\begin{description}[noitemsep]
\item [Min\_gens:]
Computes a minimal ordered set of generators $G$ of the image of $A$.
More precisely, the generators are graded, so that this
encodes the image of $A_{\leq p}$ for every~$p$.
\item [Ker\_basis:] Computes a basis $K$ of the kernel of $B$.
More precisely, the basis elements are graded, so that the basis
encodes the kernel of $B_{\leq p}$ for every~$p$.
\item [Reparam:] 
Re-expresses every element of $G$
as a linear combination in $K$, keeping its grade. 
This is possible since $BA=0$.
The resulting
graded matrix $M'$ is a semi-minimal presentation of the persistence module.
\item [Minimize:] Identifies pairs $(g,r)$ of generators (rows in $M'$)
and relations (columns in $M'$)
where $r$ ``eliminates'' $g$ and both $r$ and $g$ have the same grade. 
In that case, row $g$ and column $r$ are removed from $M'$ after some
algebraic manipulations without changing the persistence module.
Removing all such pairs results in a minimal presentation $M$.
\end{description}

We need to describe the algorithmic details of every step in order to explain our
contributions. We also provide pseudocode in Appendix~\ref{app:pseudocode} 
to complement our textual description, and will frequently refer to it.
We assume that every $x$-coordinate of a grade in $A$ and of $B$ is in $\{1,\ldots,X\}$ 
and every $y$-coordinate in $\{1,\ldots,Y\}$. Hence, we can visualize $A$ and $B$ via a $X\times Y$ integer grid, where each grid cell contains
a (possibly empty) sequence of matrix columns. Traversing the grid
row by row upwards yields the co-lexicographic order of the matrix. 
We will phrase the algorithms for
\texttt{Min\_gens} and \texttt{Ker\_basis} with this interpretation.

As an example, consider the graded matrix
\[
\Ex_1:=\begin{array}{cccccc}
A & B & C & D & E & F\\
(1,1) & (1,1) & (3,2) & (1,3) & (2,3) & (2,3)\\
\hline
0&1&0&1&0&0\\
0&1&0&0&0&0\\
1&0&0&0&1&1\\
0&0&1&0&0&1\\
0&0&1&0&1&0
\end{array}
\]
(we skip the row grades for simplicity).
Its representation as a grid is depicted in Figure~\ref{fig:grid_illustration}.

\begin{figure}
\centering
\includegraphics[width=4cm]{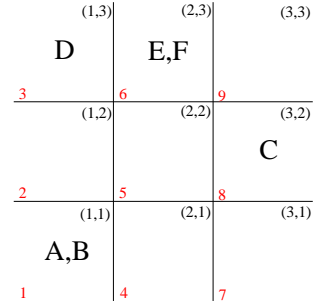}
\caption{The grid representation of the matrix $\Ex_1$ from above. The red numbers
(lower left corner per cell) show the order in which the cells are visited
by the algorithms \texttt{Min\_gens} and \texttt{Ker\_basis}.}
\label{fig:grid_illustration}
\end{figure}

\paragraph{Details of \texttt{Min\_gens}.}
The procedure traverses the columns of $A$ in a certain order defined below.
During the traversal, it maintains a \emph{pivot map} $\rho$, 
a partial map from row indices to column indices. The interpretation is that
$\rho(i)=j$ if column $j$ has been visited, has pivot $i$, and there is
no visited column $j'$ with pivot $i$ and $j'<j$. 
Initially, $\rho$ is the empty map, reflecting the state that no column
has been visited.

At any point of the algorithm, \emph{reducing} a column $j$ means
the following operation (Alg.~\ref{alg:reduce_lw}): as long as $j$ is not empty, has pivot $i$
and $\rho(i)=j'$ with $j'<j$, add column $j'$ to column $j$.
This results in cancellation of the pivot (since the coefficients are over $\Z_2$)
and hence, after the addition,
the pivot of $j$ is strictly smaller than $i$ (or the column is empty).
In either case, the reduction terminates after finitely many iterations, 
and column $j$ is marked as visited. If it ends with a non-empty
column with pivot $i$, set $\rho(i)\gets j$.

We can now describe the procedure \texttt{Min\_gens} (see Alg.~\ref{alg:min_gens_lw}):
Using the grid interpretation from above, traverse the grid cells
in lexicographic order; 
that means, the grid is traversed
column by column from the left, traversing each column bottom-up.
When reaching grid cell $(x,y)$, iterate through all
matrix columns with grade $(1,y), (2,y),\ldots, (x-1,y)$ in that order
(i.e., through all cells on the left of $(x,y)$)
and reduce them as described above. 
Then, iterate through the matrix
columns at grade $(x,y)$ and reduce them as well. 
Append every column at grade $(x,y)$ 
not reducing to $0$ in the output matrix, with grade $(x,y)$.
%(which is the grade of the currently visited column in $A$).

\ignore{
Note that because the columns of $A$ are stored in co-lexicographic order,
the columns with grades $(1,y),\ldots,(x,y)$ form one consecutive block
of columns which can be determined efficiently. In fact, by precomputing
one index per grid cell, 
%(in linear time with respect to the number of columns
%and the number of grid cells), 
the index range for iteration $(x,y)$
can be computed in constant time.
}

\paragraph{Details of \texttt{Ker\_basis}.}
This procedure is similar to the previous one, as it 
visits the columns in the same order, and reduces them when visiting.
There is one difference in the reduction procedure, however: every column
maintains an \emph{auxiliary vector}.
Initially, the auxiliary vector of column $j$ is just the unit vector $e_j$
and whenever column $j'$ is added to column $j$, we also add the auxiliary vector of $j'$
to the auxiliary vector of $j$.
In linear algebra terms, the auxiliary vectors yields an auxiliary matrix $S$
(which is the identity matrix initially)
and letting $B'$
denote the matrix arising from $B$ at any point of the algorithm, 
we maintain the
invariant that $B'=BS$
(in \cite{lw-computing}, the auxiliary matrix is called the ``slave matrix'').
In particular, if the $j$-th column of $B'$ is $0$, the $j$-th column
of $S$ encodes the linear combination of the columns of $B$
that represents a kernel element of the linear map $B$.

We describe the procedure \texttt{Ker\_basis} (Alg.~\ref{alg:ker_basis_lw}): 
Traverse the grid cells in lexicographic order.
When reaching grid cell $(x,y)$, iterate through all
matrix columns with grade $(1,y), (2,y),\ldots, (x,y)$ (in that order)
and reduce them as described above.
If any of these columns turn from non-zero to zero during the reduction,
append the auxiliary vector of the column to the output matrix
and set the grade of this column to $(x,y)$.
%(which is not necessarily the grade of the currently visited column in $B$).
The resulting matrix encodes the kernel basis of $B$. Its rows
correspond to the columns of $B$ and thus inherit their grades, yielding
a graded matrix as output.

As an example, we apply the algorithm to the matrix $\Ex_1$ in Figure~\ref{fig:grid_illustration}.
When visiting grade $(2,3)$, column $C$ is not visited yet, so the algorithm
update $\rho(5)\gets E$ and $\rho(4)\gets F$. When reaching grade $(3,2)$,
visiting column $C$, it does not perform a column operation because
$E$ is after $C$ in colexicographic order. So, $\rho(5)\gets C$ is set and
the algorithm continues. At grade $(3,3)$, the algorithm re-reduces columns
$D$, $E$, and $F$ again. In particular, when reducing $E$, it updates 
$E\gets C+E$, updating its pivot to $4$, and setting $\rho(4)\gets E$.
Then when updating $F$, it updates $F\gets E+F$, which is equal to $0$.
The corresponding auxiliary vector is $(0,0,1,0,1,1)^T$. 
Hence, the result is a kernel element representing $C+E+F$,
at grade $(3,3)$.

\paragraph{Details of \texttt{Reparam}.}
Let $G$ denote the result of \texttt{Min\_gens} and $K$
the result of \texttt{Ker\_basis}. Note that $G$ and $K$ have
the same number of rows, with consistent grades. Form the matrix
$(K|G)$ and reduce each column of $G$, using auxiliary vectors. It is guaranteed
that this turns the matrix into $(K|0)$, and the auxiliary vectors of the 
columns of $G$ yields a graded matrix $M'$ which is the output of the
procedure (Alg.~\ref{alg:reparam}).

\paragraph{Details of \texttt{Minimize}.}
Let $n$ denote the number of columns of $M'$, the output of the previous step.
Traverse the columns of $M'$ from index $1$ to $n$.
If column $i$ is a local column (i.e., the grade of its pivot coincides
with the column grade), let $j$ denote its
pivot and iterate through the columns $i+1$ to $n$; if any column $k$ contains
row index $j$, add column $i$ to column $k$ (eliminating the row index at $j$).
At the end of this inner loop, no column except $i$ has a non-zero entry
at index $j$.
%As we explain in Appendix~\ref{app:min_pres_algebra},
We can therefore remove column $i$ and row $j$ from the matrix,
without changing the persistence module that $M'$ presents.
So, remove column $i$ and row $j$ from the matrix.
After the outer loop has finished, re-index the remaining rows and columns, 
and return the resulting graded matrix $M$ as the minimal presentation (Alg.~\ref{alg:minimize_lw}).

As an example, consider the graded matrix
\[
\Ex_2:=\begin{array}{cc|ccccc}
& & A & B & C & D & E\\
& & (1,1) & (2,1) & (3,1) & (3,3) & (3,3)\\
\hline
S & (1,1)&0&1&0&0&1\\
T & (1,1)&1&0&1&0&1\\
U & (1,1)&1&1&1&1&0\\ 
V & (3,1)&0&0&1&1&0\\
W & (3,3)&0&0&0&1&1
\end{array}
\]
The algorithm identifies $A$ as a local column and adds it to columns $B$, $C$, and $D$.
$B$ is a non-local column on which the algorithm does nothing. $C$ is local
also after adding $A$, and the (modified) $C$ is added to $D$. Since $D$ is still local,
it gets added to $E$, turning $E$ into a non-local column. In the end,
we obtain
\[
\Ex_2':=\begin{array}{cc|ccccc}
& & A & \textbf{B} & C & D & \textbf{E}\\
& & (1,1) & (2,1) & (3,1) & (3,3) & (3,3)\\
\hline
\textbf{S} & (1,1)&0&1&0&0&1\\
\textbf{T} & (1,1)&1&1&0&1&0\\
U          & (1,1)&1&0&0&0&0\\ 
V          & (3,1)&0&0&1&0&0\\
W          & (3,3)&0&0&0&1&0
\end{array}
\]
and the minimal presentation is obtained by removing rows $U$, $V$, and $W$
and columns $A$, $C$, and $D$.

\paragraph{Performance.}
The algorithm described in this section is implemented 
in the RIVET library\footnote{\url{https://github.com/rivetTDA/rivet}}.
We have implemented our own version of the algorithm,
for the sake of easier integration of the improvements that follow.
Our algorithm produces precisely the same output file as RIVET
in all tested instances.
Table~\ref{tbl:rivet-comparison} shows the runtime and memory comparison
of the two implementations. 
%We observe an improvement of a factor of roughly $4$
%both in terms of time and memory, probably due to 
%low-level implementation differences.
We also compared for other types of input, and the two versions consistently
vary by a factor of 2-10 in speed, and 1-8 in memory,
probably due to low-level implementation differences.
Therefore, we consider our implementation as a valid proxy for RIVET,
and omit further comparisons with RIVET for the rest of the paper.\footnote{We remark
that we run RIVET sequentially in our tests, but parallelizing does not change the outcome
significantly.}

\begin{table}\centering
\begin{tabular}{c|cc|cc}
&\multicolumn{2}{c|}{RIVET} & \multicolumn{2}{c}{Our RIVET ``clone''}\\
n & Time & Memory & Time & Memory\\
\hline
7.5K & 15.6 & 3.10GB & 3.76 & 468MB\\
15K & 62.5 & 12.3GB & 17.6 & 1.82GB\\
30K & 353  & 49.3GB & 71.6 & 7.17GB\\
60K & --   & $>$64GB &  385 &  28.4GB \\
\end{tabular}
    \caption{Time (in seconds) and memory consumption of the algorithm in RIVET, 
and our version of it (``Clone'') for the convex hull dataset. 
The letter $K$ stands for thousands.
For $60K$ points, RIVET ran out of memory, and our version used up
around 45\% of the memory.}
\label{tbl:rivet-comparison}
\end{table}

%%%%%%%%%%%%%%%%%%%%%%%%%%%%%%%%%%%%%%%%%%%%%%%%%%%%%%%%%%%%%%%%%%%%%%%%%%%

\section{Improvements.}
\label{sec:improvements}

We describe methods to significantly improve several steps of the LW-algorithm.
We introduce them one by one and show the effect on the algorithmic performance
using the running example from last section.

\paragraph{Queues.}
From Table~\ref{tbl:rivet-comparison}, we can easily infer
a quadratic complexity of the algorithm in the size of the initial
presentation, both in time and memory.
The reason for that, however, 
is not the cost of the performed column operations
but the sheer size of the grid used in  \texttt{Min\_gens} and \texttt{Ker\_basis}. Since the points on the convex hull are chosen randomly,
the size of the grid is roughly $n\times n$, and both our algorithm
and the RIVET version store the grid as a two-dimensional array
containing some data for each grid cell, leading to this quadratic
behavior. We point out that even if we avoid constructing this array,
just iterating over every grid cell will still lead to quadratic runtime.
The obvious ``solution'' of defining a coarser grid and snapping
each grade to the closest coarse grid cell is not satisfying because
it only offers an approximate solution; on top of that, it is not clear
how to choose the coarser grid. 
We show that coarsening is not necessary, as the algorithm can be
adapted to be indifferent about the grid size in terms of performance:

First note that although the grid is of size $n\times n$ in our example,
only $O(n)$ of the grid cells appear as grades of columns of the matrix.
On the other hand, it is not sufficient to only consider these grades 
in \texttt{Min\_gens} and \texttt{Ker\_basis}.
For instance, columns of $B$ on grades $(x',y)$ and $(x,y')$ with
$x'<x$ and $y'<y$ might combine into a kernel element at grade $(x,y)$,
so \texttt{Ker\_basis} has to perform work at grade $(x,y)$ even if
no column exists at this grade. An example for this is the matrix
$\Ex_1$ from Figure~\ref{fig:grid_illustration}, where a kernel
element appears on grade $(3,3)$ as a combination of $C$, $E$, and~$F$.

The main observation is that 
we can predict the grades on which the algorithm has
to (potentially) perform operations, 
effectively avoiding to iterate through all
grid cells. Surely, every grade that appears as a grade of matrix columns 
must be considered, to visit these columns for the first time. 
Moreover, consider the situation that the algorithm is at grade $(x,y)$ and reduces a column with index $i$.
Assume further that the pivot $j$ of $i$
appears already in $\rho$ for an index $k>i$. In that case, the LW-algorithm 
updates $\rho(j)\gets i$ and stops the reduction.
However, we know more: the next time that column $k$ is visited,
column $i$, or perhaps some other column with pivot $j$, will be added
to $k$. When is this next time? Since $i<k$ and the columns are in colex order,
we know that $y$, the $y$-grade of $i$, is smaller or equal $y'$, 
the $y$-grade of $k$. If $y'=y$, column $k$ will be handled in the same iteration, and nothing needs to be done. If $y'>y$, we know that the algorithm
needs to consider grade $(x,y')$.
In the example $\Ex_1$,
this case appears when considering column $C$ at grade $(3,2)$,
having the same pivot as $E$ at grade $(2,3)$, which means that grade $(3,3)$
needs to be considered.

Based on this idea, we set up a priority queue that stores the grid cells
that need to be visited, in lexicographic order. The queue is initialized
with the column grades of the matrix. 
Then, instead of iterating over all grid cells, the algorithms \texttt{Min\_gens}
and \texttt{Ker\_basis} keep
popping the smallest element from the queue until the queue is empty,
and proceed on each grade as described before.
We extend the reduction method of a column as follows (Alg.~\ref{alg:reduce_new}):
Whenever the algorithm encounters
a situation as above during a column reduction, it pushes $(x,y')$ 
to the queue.
Every element pushed to the queue is necessarily lexicographically
larger than the current element, so the algorithm terminates~-- in the worst
case after having handled every grid cell once, 
but skipping over many grid cells in practice.
In the example $\Ex_1$, the queue would be initialized with the
four grades $(1,1),(3,2),(1,3),(2,3)$, and only the grade $(3,3)$
would be pushed into the queue during the run.

A further improvement is based on a very similar idea:
note that when a grade $(x,y)$ is handled.
both \texttt{Min\_gens} and \texttt{Ker\_basis}
still scan through all columns of grade $(x',y)$ with $x'\leq y$
and reduce all columns in this range.
Since only a few columns in this range typically need an update,
most of the time in the algorithm is wasted for scanning through this range.

Necessary updates can be predicted during earlier steps in the algorithm:
as above, when, at grade $(x,y)$ 
column $i$ is reduced and its pivot is found in a column $k>i$,
we know that column $k$ needs an update. 
Let $y'\geq y$ be the $y$-grade of $k$.
We can just remember the index $k$ and handle it the next time
when $y$-grade $y'$ is visited (which will be for grade $(x,y')$). 

Technically, we realize this idea by storing one priority queue per $y$-grade.
In the extended column reduction, in a situation as above,
the index $k$ is pushed to the priority queue of its $y$-grade.
When handling a grade $(x,y)$, instead of 
scanning through the columns of grade $(1,y),\ldots,(x-1,y)$,
we keep popping the smallest index from
the priority queue of $y$ and reduce the column (this might introduce new
elements to the priority queue, if $y=y'$ with the notation from above, but new elements are of larger index, so the procedure
eventually empties the queue). After the queue is empty, the algorithm proceeds
with the columns on grade $(x,y)$ as in the LW-version. See Alg.~\ref{alg:min_gens_new}
and~\ref{alg:ker_basis_new} for pseudocode.

It is not difficult to see that these variants perform exactly the same
column operations as the original versions of \texttt{Min\_gens} and 
\texttt{Ker\_basis} (see Appendix~\ref{app:min_pres_algebra}). 
For instance, in the example matrix $\Ex_1$, visiting column $C$
triggers to push column $E$ in the priority queue for $y$-grade $3$.
At grade $(3,3)$, $E$ is popped from the queue (skipping column $C$),
and reducing $E$ leads to pushing column $F$ into the queue, which in turn
is reduced to $0$ as in the original algorithm.

In Table~\ref{tbl:queues}, we see a substantial improvement
both in time and memory when using priority queues.
In particular, the memory consumption is close to linear.
\begin{table}\centering
\begin{tabular}{c|cc|cc}
&\multicolumn{2}{c|}{Our RIVET ``clone''} & \multicolumn{2}{c}{Using queues}\\
n & Time & Memory & Time & Memory\\
\hline
7.5K & 3.76 & 468MB & 0.70 & 32MB\\
15K & 17.6 & 1.82GB & 2.72 & 72MB\\
30K & 71.6 & 7.17GB & 10.8 & 153MB\\
60K & 385 &  28.4GB & 54.6 & 370MB\\
\end{tabular}
    \caption{Running time and memory comparison
for the original algorithm and the variant using queues to avoid extensive scanning.}
\label{tbl:queues}
\end{table}

\ignore{
The advantage of the smart scanning method is larger the larger the grid of grades
is in $x$-direction (because the LW-algorithm scans horizontally at each grid element).
Therefore, the effect is very distinguished in the running example where the
number of $x$-coordinates is large. There is a simple improvement for the LW-algorithm
to just swap $x$- and $y$-coordinates in such cases; after that swap, 
we could not observe an advantage of smart scanning anymore for this example.
However, swapping the coordinates gets less effective the more square-like the grid becomes;
our smart scanning method, in contrast, is agnostic to the shape of the grid.
}

\paragraph{Lazy minimization.}
Despite the improvement using queues, Table~\ref{tbl:queues}
shows that the practical runtime complexity is still quadratic.
The reason for that lies in the minimization procedure:
In Table~\ref{tbl:lazy}, we show the running time of the sub-step
separately, and it is evident that it is the computational bottleneck
of the algorithm (when using the queue-improvement).

Recall that in the LW-minimization, whenever it identifies
a local pair, the algorithm scans to the right to eliminate the local row index.
This scan is responsible for the observed quadratic time complexity.
Typically, the row index only appears in a few columns, and the scan
will query many columns that are not updated. Additionally, looking for a fixed
row index is a non-constant operation for most representations of column data.
For instance, if the column is realized as a dynamic array, it requires
a binary search per column.

Our improvement avoids both scanning and binary search by not
eliminating local row indices immediately when they are identified as local
(this explains the name ``lazy''). Instead, we first determine all local row and column indices 
and then we remove them ``in bulk'' in a second step. 
As we prove in Appendix~\ref{app:min_pres_algebra}, 
despite performing a different set of
column additions in general, the resulting presentation matrix is still minimal
(this optimization has been already hinted at
in~\cite[Remark 4.4]{lw-computing}). The detailed algorithmic description
follows (Alg.~\ref{alg:minimize_new}):

Let $M'$ denote the semi-minimal presentation from the \texttt{Reparam} step.
Traverse its columns in increasing index order.
At column $i$, reduce the column (by adding columns from the left that remove its pivot),
but stop the reduction as soon as the column is not local anymore. 
If column $i$ is still local
after that reduction, let $j$ denote its pivot and label column $i$ and row $j$
as local. Then proceed with the next column.

Afterwards, iterate through all columns not labeled as local.
Traverse the non-zero indices of a column $i$, in decreasing order.
If a row index $j$ is not labeled as local, keep it and proceed with the
next (smaller) index. If $j$ is local, set $k\gets\rho(j)$. Add column
$k$ to column $i$, removing the local index $j$. This operation will not
change the row indices $>j$ because $j$ is the pivot of column $k$.
Proceed with the next row index smaller than $j$. At the end of this procedure,
column $i$ does not contain local row indices anymore. At the end of the
entire procedure, all non-local columns only contain non-local row indices.
Remove all local columns and local rows from $M'$, re-index rows
and columns, and return the 
resulting matrix as minimal presentation.

We apply our algorithm on the example matrix $\Ex_2$ from before.
In the first step of the algorithm, columns $A$, $C$, and $D$ get labeled
as local, as well as rows $U$, $V$, and $W$. Moreover, column $D$ is added to $E$, resulting
in $E\gets (1,1,1,1,0)^T$. The reduction stops because $E$ is not local anymore
after this addition.

In the second iteration, the algorithm re-visits the non-local columns $B$ and $D$.
For $B$, which has pivot $U$, it adds columns $A$, because ($(U,A)$ is a local pair).
This results in $(1,1,0,0,0)^T$, and no more column addition is performed
on $B$ because $S$ and $T$ are both non-local rows. For column $E$ (which has pivot $V$
at this point), column $C$ is added, resulting in $E=(1,0,0,0,0)$,
and the algorithm stops because $S$ is local. Note that the resulting columns $B$
and $E$ are the same as in the matrix $\Ex_2'$ with the LW-procedure.

Table~\ref{tbl:lazy} shows the effect of the new minimization procedure, 
essentially turning the minimization step from the computational bottleneck
to an operation whose running time is negligible. The memory consumption
of the algorithm does not change significantly, so we omit the numbers.
Interestingly, we observed that the total number of column additions
performed by the lazy minimization is typically slightly larger than in the original method
(unlike in our example matrix $\Ex_2$ above).
This underlines that the observed speed-up is due to avoiding scanning
the matrix repeatedly and not due to saving matrix operations. 

\begin{table}\centering
\begin{tabular}{c|cc|cc}
&\multicolumn{2}{c|}{LW-minimization} & \multicolumn{2}{c}{Lazy minimization}\\ 
n & Total & Minimize & Total & Minimize\\
\hline
7.5K & 0.70 & 0.56 & 0.15 & 0.00\\
15K & 2.72 & 2.40 & 0.33 & 0.01\\
30K & 10.8 & 10.1 & 0.73 & 0.02\\
60K & 54.6 & 53.1 & 1.62 & 0.07\\
\end{tabular}
    \caption{Running times (in seconds) of the algorithm with the original
and the lazy minimization procedure (both versions use the queue-optimization).}
\label{tbl:lazy}
\end{table}

\paragraph{Chunk preprocessing.}
We use the multi-chunk algorithm from~\cite{fk-chunk} as initial step
of our minimal presentation algorithm.
Strictly speaking, this is not an improvement of the LW-algorithm, as
it only replaces the input matrices $(A,B)$ with (smaller) matrices $(A',B')$
that yield an isomorphic persistence module. The basic idea is to identify
row-column-pairs in $A$ whose removal does not affect the persistence module,
just as in the minimization step. However, we try to do so early
in the process, hopefully removing a significant number of pairs
upfront, so that the subsequent steps of the algorithm can operate
on smaller matrices.

The details of the chunk algorithm have already been described 
in the lazy minimization step from above: 
just replace the matrix $M'$ with $A$.
Additionally, when removing row $j$ from $A$ at the end of algorithm,
remove column $j$ from $B$ as well.
After re-indexing, we obtain matrices
$(A',B')$ which are the output of the chunk reduction.

We see in Table~\ref{tbl:chunk_et_al} (first two lines per instance) 
that the chunk reduction is fast
and improves the performance of all subsequent steps, resulting in a speed-up of roughly
a factor of $2$ in this example. Perhaps more importantly, also the memory usage drops
by a factor $4$ (with both factors getting gradually better for
larger instances). 
The total number of columns of $A'$ and $B'$
is consistently around a quarter of those of $(A,B)$ 
which is in accordance to the memory savings.

\begin{table*}[h]\centering
\begin{tabular}{c|l|cccccc|cc|c}
n & Variant & IO & Ch & MG & KB & RP & Min & Time & Mem & Size\\
\hline
\multirow{3}{*}{200K}
& queue,lazy & 2.67 & - & 0.71 & 2.06 & 0.67 & 0.34 & 6.61 & 1.53GB & \multirow{3}{*}{(477.6, 476.6)}\\
& +chunk & 2.50 & 0.23 & 0.06 & 0.45 & 0.06 & 0.01 & 3.38 & 421MB\\
& +parfor & 2.53 & 0.22 & 0.06 & 0.46 & 0.03 & 0.01 & 3.38 & 420MB\\
\hline
\multirow{3}{*}{400K}
& queue,lazy & 5.55 & - & 1.70 & 4.74 & 1.53 & 0.81 & 14.5 & 3.49GB & \multirow{3}{*}{(698.4, 697.4)}\\
& +chunk & 5.18 & 0.49 & 0.18 & 1.07 & 0.13 & 0.03 & 7.19 & 835MB\\
& +parfor & 5.19 & 0.44 & 0.18 & 1.08 & 0.07 & 0.02 & 7.10 & 834MB\\
\hline
\multirow{3}{*}{800K}
& queue,lazy & 11.2 & - & 4.01 & 10.8 & 3.51 & 1.94 & 32.0 & 7.78GB & \multirow{3}{*}{(1022.0, 1021.0)}\\
& +chunk & 10.4 & 0.99 & 0.50 & 2.47 & 0.28 & 0.08 & 14.9 & 1.62GB\\
& +parfor & 10.4 & 0.88 & 0.49 & 2.49 & 0.13 & 0.05 & 14.7 & 1.61GB\\
\hline
\multirow{3}{*}{1.6M}
& queue,lazy & 21.9 & - & 9.36 & 24.9 & 8.54 & 4.84 & 70.8 & 17.7GB & \multirow{3}{*}{(1520.6, 1519.6)}\\
& +chunk & 20.3 & 2.00 & 1.40 & 5.60 & 0.64 & 0.20 & 30.6 & 3.07GB\\
& +parfor & 20.1 & 1.74 & 1.36 & 5.64 & 0.30 & 0.17 & 29.7 & 3.12GB\\
\end{tabular}
    \caption{Time and memory consumption 
of the algorithm (on the convex hull dataset) regarding chunk optimization and parallelization of for-loops.
Columns from left to right: Input size, algorithm that was used, running times in seconds for
setting up the graded matrices from input, and writing the result to output (IO), chunk preprocessing (Ch),
min-gens (MG), ker-basis (KB), reparameterization (RP), minimization (Min), total time (Time), memory consumption (Mem),
and the size of the output (Size). In the ``Variant'' column, the additions are cumulative, e.g, ``+chunk'' means that the option 
is used also in the subsequent line.}
\label{tbl:chunk_et_al}
\end{table*}

\paragraph{Parallelization.}
The main loops in the procedures \texttt{chunk\_preprocessing}, \texttt{reparam}, and \texttt{minimize} can be easily parallelized in shared memory:
removing the local row indices of a non-local column
only requires reading access to local columns and can be performed
independently for every non-local column. Similarly, re-expressing
a generator in terms of a kernel basis in \texttt{reparam}
can be done for all generators in parallel.

In Table~\ref{tbl:chunk_et_al}, the third line per instance shows the effect
of parallelization in the examples 
(we used \textsc{OpenMP} for the parallelization). 
We can see only a marginal effect for large instances.
However, we will show in Section~\ref{sec:further_experiments}
that parallelization brings a more substantial speed-up in other types
of instances.

\ignore{

ce of a slightly higher memory usage). 
Given that we run the algorithm in parallel with $12$ cores,
we could hope for an even better improvement, but at least, the experiment shows that
parallelization has a measurable effect.
}

The procedures \texttt{Min\_gens} and \texttt{Ker\_basis} do not permit
an easy parallelization scheme, but since they act independently
on the matrices $A$ and $B$, they can be run in parallel
on two cores (of course, yielding a speed-up factor of at most $2$). We do not provide tabular data
for this version, but it can be easily derived from the ``+parfor'' lines in Table~\ref{tbl:chunk_et_al},
just subtracting the time in the ``MG'' column (the memory consumption is unaffected by
this parallelization).

\paragraph{Clearing.}
We implemented one more optimization for \texttt{Ker\_basis}
reminiscent of the clearing optimization
in the single parameter case~\cite{ck-twist,bkr-clear}.
However, this clearing optimization takes no effect in our running
example. In Appendix~\ref{app:clearing}, we describe the optimization
and present one example with a (modest) speed-up.

%%%%%%%%%%%%%%%%%%%%%%%%%%%%%%%%%%%%%%%%%%%%%%%%%%%%%%%%%%%%%%%%%%%%%%%%%%%

\section{Further experimental evaluation.}
\label{sec:further_experiments}

In this section, we summarize experimental comparisons between the LW-algorithm
and our improved version. We consider many types of input,
since the asymptotic behavior of both the LW-algorithm and our improved version
can vary significantly across different input types.
For the function-Rips data (explained below), our improvements led to a constant factor
runtime improvement of around $10$ in the largest example;
for all other types of input, our improvements led to an asymptotic runtime improvement.

\paragraph{Function-Rips data.}
We created bi-graded data sets of various sizes in the following way:
we sampled $n$ points from a noisy circle and consider the simplicial
complex consisting of all $n$ points, $e=\binom{n}{2}$ edges
and $f=\binom{n}{3}$ triangles. 
Unlike in the previous example, the resulting matrices $(A,B)$ are
very imbalanced, and the size of the input is dominated by $f$,
the number of columns of $A$.
The first coordinate of the bi-grade
is $0$ for vertices, the distance between endpoints for edges, 
and the length of the longest edge for triangles.
This is the \emph{Vietoris-Rips complex}, one of the most prominent complexes
in topological data analysis. For the second grade coordinate,
we used a kernel density estimate with Gaussian kernel with fixed bandwidth
for the vertices, and this is extended to edges and triangles
by assigning the maximal value among the boundary vertices.
The construction ensures that the subcomplex at each grade
is a \emph{flag complex}, that is, a complex that is completely
determined by its vertices and edges.

Also by construction, the number of $x$-grades in these instances
is around $e$, and the number of $y$-grades is around $n$.
This implies that the grid size is comparable
to the number of columns. 

\ignore{

\begin{table}[h]\centering
\begin{tabular}{ccc|c|c}
n & e & f & (X,Y) & OUT\\
\hline
%508K  & 498K, 10.4K, 145  & 10.4K,144 & 59$\times$36\\
1.02M & 1.00M, 16.7K, 183 & 16.7K,182 & 66$\times$47 \\
2.05M & 2.03M, 26.6K, 231 & 26.6K,229 & 108$\times$76\\
4.11M & 4.06M, 42.2K, 291 & 42.2K,290 & 252$\times$213\\
8.17M & 8.10M, 66.8K, 366 & 66.8K,365 & 365$\times$308
\end{tabular}
    \caption{The statistics of the function-Rips benchmark data.}
\label{tbl:function_rips_statistics}
\end{table}

We see the characteristics of these datasets 
in Table~\ref{tbl:function_rips_statistics}; note that the number $m$
in this table is the number of sample points. We observe that
for inputs $(A,B)$ of that kind, the columns in $A$ is much larger
than the columns in $B$ (because $\binom{p}{3}\gg\binom{p}{2}$).
Also, the number of $x$-grades corresponds to the number of edges
and the the number of $y$-grades is close to the number
of vertices; both follows directly from our construction.
Finally, we observe that the resulting minimal presentation 
is of very small size compared to the input size. This once more underlines
the importance of computing the minimal presentation as a preprocessing step
in a pipeline of algorithms.

}

\begin{table*}[h]\centering
\begin{tabular}{c|l|cccccc|cc|c}
N & Variant & IO & Ch & MG & KB & RP & Min & Time & Mem & Size\\
\hline
\multirow{4}{*}{1.02M}
& rivet\_clone & 1.14 & - & 17.4 & 0.41 & 0.02 & 1.57 & 20.6 & 482MB & \multirow{4}{*}{(72.2, 48.2)}\\
& +queue,lazy & 1.14 & - & 5.24 & 0.02 & 0.02 & 0.00 & 6.46 & 488MB\\
& +chunk & 1.16 & 3.94 & 1.51 & 0.00 & 0.00 & 0.00 & 6.64 & 828MB\\
& +parfor & 1.16 & 1.81 & 1.66 & 0.00 & 0.00 & 0.00 & 4.67 & 1.00GB\\
\hline
\multirow{4}{*}{2.05M}
& rivet\_clone & 2.34 & - & 59.0 & 1.01 & 0.04 & 5.06 & 67.6 & 1.01GB & \multirow{4}{*}{(105.6, 67.8)}\\
& +queue,lazy & 2.32 & - & 13.1 & 0.05 & 0.04 & 0.00 & 15.6 & 1.01GB\\
& +chunk & 2.33 & 9.94 & 3.83 & 0.00 & 0.00 & 0.00 & 16.1 & 1.84GB\\
& +parfor & 2.32 & 4.34 & 4.16 & 0.00 & 0.00 & 0.00 & 10.9 & 2.27GB\\
\hline
\multirow{4}{*}{4.11M}
& rivet\_clone & 4.67 & - & 182 & 2.44 & 0.07 & 14.1 & 204 & 2.19GB & \multirow{4}{*}{(123.0, 85.4)}\\
& +queue,lazy & 4.65 & - & 33.4 & 0.09 & 0.06 & 0.01 & 38.3 & 2.20GB\\
& +chunk & 4.69 & 25.1 & 9.12 & 0.00 & 0.00 & 0.00 & 39.0 & 4.05GB\\
& +parfor & 4.62 & 10.3 & 9.64 & 0.00 & 0.00 & 0.00 & 24.7 & 5.02GB\\
\hline
\multirow{4}{*}{8.17M}
& rivet\_clone & 9.36 & - & 554 & 5.90 & 0.13 & 40.5 & 610 & 4.93GB & \multirow{4}{*}{(175.6, 121.0)}\\
& +queue,lazy & 9.33 & - & 93.5 & 0.17 & 0.12 & 0.02 & 103 & 4.94GB\\
& +chunk & 9.42 & 66.0 & 26.5 & 0.01 & 0.00 & 0.00 & 102 & 9.43GB\\
& +parfor & 9.44 & 25.6 & 27.7 & 0.01 & 0.00 & 0.00 & 63.2 & 12.1GB\\
\end{tabular}
\caption{Time and memory consumption for the function-Rips dataset.
  $N:=\ell+n$ is the total number of columns in $A$ and $B$.
The meaning of the other columns is the same as in Table~\ref{tbl:chunk_et_al}.}
\label{tbl:function_rips_runtime}
\end{table*}

The experimental results are displayed in Table~\ref{tbl:function_rips_runtime}.
We observe that the bulk of running time is spent in the \texttt{Min\_gens} procedure.
We also notice that our queue optimization still has a significant effect,
but not as strong as in the example from the last section. To explain this, note that
from the $\binom{p}{3}$ columns in $A$, only $\binom{p}{2}$ can give rise to a relation;
the remaining ones will be reduced to $0$ in \texttt{Min\_gens}, and a large part
of the algorithm is spent to perform these reductions; this part will not be short-cut
by our improvements. In other words, in comparison to the previous example,
more time is spent on actual matrix operations.

We also see that using chunk-preprocessing has hardly any effect on time, but roughly
doubles the memory consumption. This is not too surprising
because the number of row/column-pairs that are removed is insignificant with respect to the
total number of columns in $A$. 
That means that the pair $(A',B')$ is of the same size as $(A,B)$, and our implementation
deletes $(A,B)$ after having generated $(A',B')$, leading to a doubling of memory 
(a more careful implementation could avoid representing both matrix pairs at the same time
and save this memory overhead).
However, some speed-ups are obtained by parallelization
because some of the unavoidable reduction steps can be performed in parallel
using chunk preprocessing, although this leads to a further increase in memory
consumption.
In summary, we see a running time improvement
of a factor of around $10$ in the biggest example.

\paragraph{Additional mesh data.}
We also tested with triangular mesh data used in~\cite{fk-chunk},
publicly available at the AIM@SHAPE repository\footnote{available at \url{http://visionair.ge.imati.cnr.it/}}. 
The graded input matrices $(A,B)$ were generated in the same way as for the sphere meshes
in Section~\ref{sec:improvements}.

The outcome confirms the results from Section~\ref{sec:improvements}: the original version
runs out of memory for all but the smallest instances, and our improved version can
handle all meshes within a few seconds, using only a fraction of the memory.
For details, see Table~\ref{tbl:mesh_runtime} in Appendix~\ref{app:experiments_appendix}.

\paragraph{Random Delaunay triangulations.}
We tested a variant of the mesh example: we sampled $n$ points inside the unit sphere,
and we computed the Delaunay triangulation of these $n$ points (using \textsc{CGAL}~\cite{cgal:delaunay}).
As before, we assign to vertices their $x$- and $y$-coordinates as grades,
and to edges, triangles, and tetrahedra using the maximal grade of the boundary vertices.
We have two choices to generate a free implicit representation $(A,B)$, either using the vertices,
edges, and triangles, or using edges, triangles, and tetrahedra. These two options
correspond to computing a representation of (homology) dimension $1$ and $2$, and we refer to them
as the $d=1$ and $d=2$ case, respectively.
%By construction, homology in the $d=2$ case is trivial,
%and the resulting minimal presentation has size $0\times 0$.

We tested both cases and list the numbers in Tables~\ref{tbl:random_del_1_small},
\ref{tbl:random_del_2_small}, \ref{tbl:random_del_1_large}, and \ref{tbl:random_del_2_large}
in Appendix~\ref{app:experiments_appendix}. 
In both cases, using queues and lazy minimization has a tremendous effect on the performance.
Remarkably, chunk preprocessing has a very different effect:
for $d=2$, it saves runtime (and yields an asymptotic improvement), whereas for $d=1$, it actually
increases the runtime. However, when parallelizing, the overall performance is better
with the chunk optimization.

\paragraph{Multi-cover filtrations}
We also tried our algorithm on bifiltrations coming from multi-covers of point clouds: 
the simplicial complex at grade $(r,k)$ represents the region of the plane
covered by at least $k$ balls of radius $r$ around the input points;
see~\cite{eo-multi} for more details.
Different from the previous examples, the grid for these examples
is very narrow, with only $10$ different grades in $y$-directions.
Despite the small-sized grid, using queues has a significant effect in practice,
as well as lazy minimization and chunk preprocessing. Interestingly, this is the only
type of example where we could see a (small) improvement when using the
clearing optimization, discussed in Appendix~\ref{app:clearing}.
We give numbers in Tables~\ref{tbl:multicover_small} and \ref{tbl:multicover_large}
in Appendix~\ref{app:experiments_appendix}.

\paragraph{Matrix data structure.}
Our implementation uses the \texttt{boundary\_matrix} data structure
of the \textsc{Phat} library to store graded matrices. One of the advantages
of that is that \textsc{Phat}'s matrix type is generic in the column representation.
In all our experiments above, we have used dynamic arrays (\texttt{vector\_vector} in \textsc{Phat}).
Our code allows one to simply exchange the column type with any of the $7$ other types that \textsc{Phat} provides~-- see~\cite{bkrw-phat} for a description of them. 
%While some of them are standard choices (e.g., \texttt{vector\_list} stores the matrix as an array
%of linked lists), others are more complicated~-- we refer to~\cite{bkrw-phat} for details.

We repeated some of the previous experiments, trying all the different column types from \textsc{Phat},
producing data similar to~\cite[Table 1--7]{bkrw-phat}. 
The result is that the \texttt{vector\_vector} column type
is the fastest in most cases, although \texttt{bit\_tree\_pivot\_column} (which is 
the default representation in \textsc{Phat}) is competitive and sometimes faster,
at least when the algorithm runs in parallel. The experimental data
is in Table~\ref{tbl:column_compare} in Appendix~\ref{app:experiments_appendix}.

\section{Conclusion.}
\label{sec:conclusion}
We described several improvements to the Lesnick-Wright algorithm
that make the computation of minimal presentations of
bi-graded persistence modules very fast in practice.
Some of our improvements are reminiscent of similar techniques
in single-parameter persistence, namely the chunk preprocessing
and the clearing method. If one is interested in minimal presentations
in all homology dimensions, another form of clearing seems possible,
where the information in dimension $(d+1)$ is used to clear out columns
from $A$ without reducing them (as in~\cite{ck-twist} for one parameter).
For Vietoris-Rips filtrations, this works well together with dualizing
the filtration, that is, computing persistent cohomology instead~\cite{smv-duality},
culminating in specialized very fast implementations for Vietoris-Rips complexes~\cite{ripser}.
It is not clear how this ``duality trick'' can be carried over
to the multi-parameter setup.

\ignore{
While we demonstrated that parallelization can yield to speed-ups, the scaling factor
is rather disappointing (a factor of around $3$ at best with $12$ cores). 
We only used straight-word for-loop-parallelization using the \textsc{OpenMP} library;
perhaps the scaling can be improved with more sophisticated parallelization schemes.
}

We plan to integrate our algorithm into a larger pipeline for 2-parameter persistence.
In this, the minimal presentation is preceded by a method to generate
a bi-graded simplicial complex (possibly with approximate methods), and succeeded by
an algorithm to analyze the data set that the presentation represents.
We want to show that passing to a minimal presentation is a crucial step to
make such a pipeline feasible for realistic problem sizes.
The last columns of Tables~\ref{tbl:chunk_et_al} and \ref{tbl:function_rips_runtime}
show that minimal presentations are typically small compared to the input size,
giving us hope that subsequent algorithms will profit a lot from this simplification.

A further natural extension is to design a minimal presentation algorithm
for $3$ and more parameters.
An obstacle to this extension is that, in algebraic terms,
the kernel of a map between free modules need not be free for $3$ or more parameters.
This fact is crucial for the correctness of the \texttt{ker\_basis} method.
Dey and Xin~\cite{dx-generalized} describe
an algorithm to compute a presentation in the general case, but not necessarily
a minimal one, and an efficient implementation is missing.

Our approach is an instance of ``lossless compression'': the 
persistence module encoded in the output is isomorphic to the one of the 
input. In the single-parameter case, a popular research direction 
is the approximation of persistence diagrams, where the input is reduced
in a way that the resulting module is provably close to the original one
in terms of bottleneck distance (e.g.,\cite{sheehy-linear,dfw-computing,ckr-improved}). 
The prospects of this idea are unexplored in the multi-parameter setup.

\bibliographystyle{plain}
\bibliography{bib}

\clearpage

\begin{appendix}

\section{Simplicial bifiltrations: a running example}
\label{app:simplicial}

For readers not familiar with some of the concepts underlying this work,
we develop a running example on which we are illustrating most of the
basic concepts needed in this work. We are rather verbose here,
and readers familiar with basics of computational topology can
probably read over the first subparagraph quite quickly.

\paragraph{Simplicial complexes and homology}
A \emph{simplicial complex} $C$ over a finite vertex set $V$
is a collection of non-empty subsets of $V$ that is closed
under taking subsets. Elements of the complex of cardinality
$k+1$ are called \emph{$k$-simplices}, and $k$ is the dimension
of the simplex. Simplices of dimension $0$, $1$, and $2$,
are called \emph{vertices}, \emph{edges}, and \emph{triangles},
respectively. The dimension of a simplicial complex $C$
is the maximal dimension of its faces.
In Figure~\ref{fig:complex}, we see a $2$-dimensional simplicial
complex over $V=\{A,B,C,D,E\}$, consisting of $5$ vertices,
$7$ edges and $2$ triangles. Clearly, this is generalizing the notion
of a graph, as a graph is a simplicial complex of dimension $1$.
We call a $k$-simplex $\tau$ a \emph{facet} of a $(k+1)$-simplex $\sigma$
if $\tau\subset\sigma$.

\begin{figure}
\centering
\includegraphics[width=3cm]{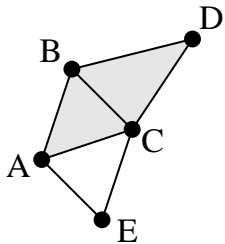}
\caption{Example of a simplicial complex.}
\label{fig:complex}
\end{figure}

Fixing a simplicial complex $C$ and a dimension $k$, 
the \emph{$(k+1,k)$-boundary matrix} is a matrix where every column
corresponds to a $(k+1)$-simplex of $C$, every row corresponds
to a $k$-simplex of $C$, and the $(i,j)$-entry is $1$ if
the $k$-simplex in row $i$ is a facet of the $(k+1)$-simplex
in column $j$. The simplicial complex depicted in Figure~\ref{fig:complex}
is fully described by the pair $(A,B)$ with $A$ its $(2,1)$-boundary matrix
and $B$ its $(1,0)$-boundary matrix, given as
\[A=\begin{array}{c|cc}
& ABC & BCD\\
\hline
AB & 1 & 0\\
AC & 1 & 0\\
AE & 0 & 0\\
BC & 1 & 1\\
BD & 0 & 1\\
CD & 0 & 1\\
CE & 0 & 0
\end{array}
\]
\[
B=\begin{array}{c|ccccccc}
& AB & AC & AE & BC & BD & CD & CE\\
\hline
A & 1 & 1 & 1 & 0 & 0 & 0 & 0\\
B & 1 & 0 & 0 & 1 & 1 & 0 & 0\\
C & 0 & 1 & 0 & 1 & 0 & 1 & 1\\
D & 0 & 0 & 0 & 0 & 1 & 1 & 0\\
E & 0 & 0 & 1 & 0 & 0 & 0 & 1
\end{array}.
\]
Generally, a simplicial complex
of dimension $d$ is described by $d$ such matrices.

Besides representing the complex, the matrices can be used to obtain further
information using linear algebra. For instance, the union of the two triangles
$ABC$, $BCD$ can be written as the vector $v:=(1,1)^T$, and matrix-vector
multiplication yields $Av=(1,1,0,0,1,1,0)$, corresponding to the edges
$AB$, $AC$, $BC$, $BD$, which are precisely the four edges bounding
the quadrilateral formed by the union of the two triangles.
Similarly, the three edges $AC$, $BC$, $BD$ form a path, and using the
vector $w:=(0,1,0,1,1,0,0)^T$, we get $Bw=(1,0,0,1,0)^T$, corresponding to 
the vertices $A$ and $D$, which are the endpoints of the path.
Note that in these examples, we exploit the fact that we interpret our
matrices to be over $K=\Z_2$; a generalization is possible, but requires
introducing orientations of simplices.

Using the three edges $AB$, $AC$, $BC$, we see that the corresponding vector $z_1$
satisfies $Bz_1=0$. This makes sense because the three edges form a cycle in the graph,
hence there is no ``boundary''. The same is true for the vector $z_2$ 
formed by $BC$, $BD$, $CD$, and this implies
that the matrix product $B\cdot A$ is the zero matrix. This is not a coincidence,
but the so-called \emph{fundamental lemma of homology}, stating this to be true
for every simplicial complex and and two boundary matrices in consecutive dimensions.
In other words, we have that the image of $A$ is contained in the kernel of $B$.
However, we can observe the kernel of $B$ is larger, since it also contains,
for instance, the vector $z_3$ formed by $AC$, $AE$, $CE$ (we leave the translation to the actual
vector in $\Z_2^7$ to the reader from now on).
We can furthermore observe that $\{z_1,z_2,z_3\}$ is a basis of the kernel of $B$,
which translates into the fact that every cycle in the graph can be obtained
as a symmetric difference of the $3$ corresponding cycles.
We are interested in the difference of the kernel of $B$ and the image of $A$,
given by the quotient space
\[
H_1:=\mathrm{ker} B / \mathrm{im} A
\]
which in this case is one-dimensional and can be represented
by the cycle $z_3$. This is called the \emph{homology group}
of the simplicial complex in dimension $1$. As we deal with a base
field (as opposed to integer coefficients as in the classical theory),
the homology group is in fact a vector space, but we still use the
standard notation. Intuitively, $H_1$ captures the cycles of the
simplicial complex which are ``non-trivial'', in the sense
that they are not bounded by a set of triangles. In our example,
this is clearly the case for $z_3$, but not for $z_1$ or~$z_2$.

\paragraph{Bi-filtrations and graded matrices}
We introduce more structure by giving every simplex a (bi-)grade
in $\Z^2$. This assigns a grade to each row and column of $A$ and $B$,
turning them into graded matrices. A possible assignment for
our running example would be

\[A=\begin{array}{cc|cc}
& & ABC & BCD\\
& & (2,2) & (3,3)\\
\hline
AB & (1,1) & 1 & 0\\
AC & (2,2) & 1 & 0\\
AE & (2,3) & 0 & 0\\
BC & (1,1) & 1 & 1\\
BD & (1,2) & 0 & 1\\
CD & (2,1) & 0 & 1\\
CE & (3,2) & 0 & 0
\end{array}
\]
\[
\tiny{
B=\begin{array}{cc|ccccccc}
& & AB & AC & AE & BC & BD & CD & CE\\
& & (1,1) & (2,2) & (2,3) & (1,1) & (1,2) & (2,1) & (3,2)\\
\hline
A & (1,1) & 1 & 1 & 1 & 0 & 0 & 0 & 0\\
B & (1,1) & 1 & 0 & 0 & 1 & 1 & 0 & 0\\
C & (1,1) & 0 & 1 & 0 & 1 & 0 & 1 & 1\\
D & (1,1) & 0 & 0 & 0 & 0 & 1 & 1 & 0\\
E & (2,2) & 0 & 0 & 1 & 0 & 0 & 0 & 1
\end{array}}.
\]
Perhaps more illustrative, the same data can be displayed
by a two-dimensional sequence of simplicial complexes
as depicted in Figure~\ref{fig:bifiltration}.
In here, at each grade $(x,y)$, we draw all simplices whose
grade is $\leq (x,y)$ in the defined partial order over $\R^2$.
The simplices whose grade is equal to $(x,y)$ are drawn in red.
Recall that we required that for a column with grade $g$, 
with an $1$-entry in a row with grade $g'$, $g'\leq g$ must hold.
This translates into the property that at every grade $g$, the
collection of simplices collected in this way forms a simplicial complex.
This data is called a \emph{bifiltration} of a simplicial complex
and can be used to model the evolution of data when two scale parameters
are varied independently from each other.

\begin{figure}
\centering
\includegraphics[width=8cm]{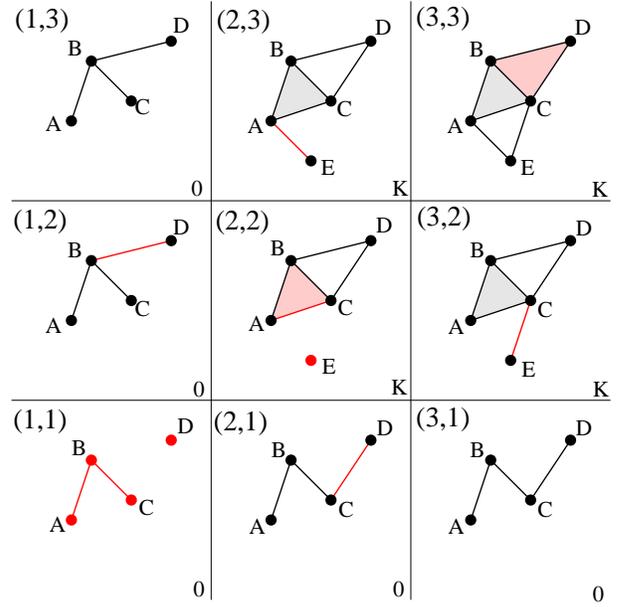}
\caption{Example of a bifiltration. The upper left corner denotes the grade,
the lower-right denotes the homology group of the complex.}
\label{fig:bifiltration}
\end{figure}

We can redo the same computation as before for any fixed grade now:
for instance, on grade $p:=(2,2)$, we obtain $A_{\leq p}$, which is $A$
from above with column $2$ and rows $3$ and $7$ removed.
$B_{\leq p}$ is $B$ with columns $3$ and $7$, and row $5$ removed.
The kernel of $B_{\leq p}$ is spanned by $z_1$ and $z_2$, the image
of $A_{\leq p}$ by $z_1$, and hence, the quotient group is one-dimensional,
represented by $z_2$. We can do similar calculations for each grade
and obtain the homology group for each grade, which is depicted in the
lower-right corners in Figure~\ref{fig:bifiltration}.
Revisiting the definition of the persistence module represented by the 
two graded matrices $(A,B)$, these homology groups are precisely the
vector spaces $V_p$ in the definition of Section~\ref{sec:definitions}.

We stress out the the persistence module contains more information
than just the homology on each grade. In some sense, the information
is encoded more in the maps between the grades than in the grades
themselves. As an example, set $p:=(2,2)$ and $p':=(3,3)$.
Recall that $\mathrm{ker} B_{\leq p}$ is generated by $z_2$.
The map $f_{p\to p'}$ is defined just by mapping $z_2$
to $z_2$ in  $\mathrm{ker} B_{\leq p'}$, which represents a trivial element
in the homology group because $z_2$ is bounded by a triangle at grade $(3,3)$.
That means that the generator at grade $(2,2)$ becomes trivial at grade $(3,3)$,
meaning that although the vector spaces are isomorphic at both grades,
the generators is changing. In contract, considering $p''=(3,2)$,
the map $f_{p\to p''}$ maps generator to generator, because $z_2$
generates the homology group at scale $(3,2)$.

We also point out that there is nothing special about the case of triangle,
edges, and vertices: the entire approach generalizes to higher dimensions
considering any two consecutive boundary matrices of a simplicial complex.

\paragraph{Conversion to a minimal presentation.}
We exemplify the conversion of $(A,B)$ as given above into a minimal 
presentation, which is the goal of our algorithm.
The main point is to identify a graded basis $Z$ of $\mathrm{ker} B$, that is,
a basis with a bi-grade for each element, such that $Z_{\leq p}$,
denoting the basis elements of $Z$ with grade $\leq p$, is a basis
of $\mathrm{ker} B_{\leq p}$ for all $p$. Such a basis always exists,
but this is non-trivial to prove (it follows from the fact that the kernel
of a map between free $K[x,y]$-modules is free, where $K[x,y]$ is the polynomial
ring with two variables). In our example, however, such a graded kernel can be
read off immediately: We pick $\{z_1,z_2,z_3\}$, where as above, $z_1$ is the 
cycle $AB$, $AC$, $BC$, $z_2$ is the cycle $BC$ $BD$, $CD$ and $z_3$ is the cycle
$AC$, $AE$, $CE$, and we assign the grades $(2,2)$ to $z_1$ and to $z_2$, 
and $(3,3)$ to $z_3$. The procedure \texttt{Ker\_basis} of Section~\ref{sec:lesnick-wright} 
yields such a graded basis in general.

Since $B_{\leq p} A_{\leq p}=0$, every column of $A$ at grade $p$ can be expressed
as a linear combination of the basis elements from above of grades $\leq p$.
In our case, this is very simple: The boundary of the triangle $ABC$ is equal to $z_1$,
and the boundary of the triangle $BCD$ is equal to $z_2$. Written as graded matrix,
this yields
\[
M'=\begin{array}{cc|cc}
& & ABC & BCD\\
& & (2,2) & (3,3)\\
\hline
z_1 & (2,2) & 1 & 0\\
z_2 & (2,2) & 0 & 1\\
z_3 & (3,3) & 0 & 0\\
\end{array}
\]
In general, this step requires to solve a linear system for every column, 
and the procedure \texttt{Reparam} of Section~\ref{sec:lesnick-wright} 
is solving these systems in general.

The matrix $M'$ is a presentation of the same persistence module as 
given by $(A,B)$. For example, at grade $p=(2,2)$, we restrict ourselves
to rows and columns of grade $\leq p$, which are column $1$
and row $1$ and $2$. Hence, using the definition from Section~\ref{sec:definitions}, 
we get
\[
V_p=\langle z_1,z_2\rangle / \langle z_1\rangle
\]
which means that $V_p$ is spanned by $z_2$. Similarly, for $q=(3,3)$,
we obtain $V_q$ to be spanned by $z_3$, and it can be readily checked
that the vector spaces, and maps in-between coincide at each grade
with the ones defined by the matrix pair $(A,B)$.

While it does not happen in our small example, it may happen that
columns of $B$ give rise to superfluous relations (=columns)
in the presentation matrix. This is the case whenever a column 
of grade $p$ can be
expressed as linear combination of other columns of grade $\leq p$.
The smallest such example would be an example including the $4$
boundary triangles of a tetrahedron on the same grade, 
where the last triangle added is just the sum of the other three.
The procedure \texttt{Min\_gens} filters out these superfluous
relations in the general case.

The matrix $M'$ is much succinct description of the persistence module
compared to $(A,B)$,
but it is still not minimal. The reason is that the kernel element $z_1$ is
created at scale $(2,2)$, but the the triangle $ABC$ is also created at $(2,2)$,
including the element into the image of $A$ (and hence trivializing it) right
away. In the example, we can just remove the first row and column and arrive at
\[
M=\begin{array}{cc|c}
& & BCD\\
& & (3,3)\\
\hline
z_2 & (2,2)  & 1\\
z_3 & (3,3)  & 0\\
\end{array}
\]
which is a minimal presentation. Indeed, re-considering Figure~\ref{fig:bifiltration},
we can summarize the situation as: there is a non-trivial cycle coming into existence
at grade $(2,2)$ which gets filled up at grade $(3,3)$, and there is another
non-trivial cycle starting at grade $(3,3)$, which is never filled up.

In general, we cannot just remove a row and column from a presentation: we first have
to ensure that every other column is re-expressed that the row to be removed
does not appear as row index. 
Mathematically, that means to change the basis of $\mathrm{im} B$; the procedure
\texttt{minimize} performs this removal in general.

\section{Correctness} \label{app:min_pres_algebra}

\paragraph{Correctness of lazy minimization.}
We first reformulate the definition of a minimal presentation of a persistence module
in terms of commutative algebra.
For more details on this perspective, see \cite{lw-computing}.
In this language, a \emph{persistence module} is a $\Z^2$-graded $K[x,y]$-module,
where $K = \Z_2$ as before,
and a \emph{presentation} of a persistence module $M$ is a homomorphism $p : F^1 \rightarrow F^0$
between free persistence modules such that $\mathrm{coker}(p) \cong M$.
We say that a \emph{basis} for a free persistence module is a minimal homogeneous set of generators.
Let $p$ be a presentation of $M$, and write $q : F^0 \rightarrow M$ for the canonical map;
the presentation is \emph{minimal} if
$(1)$ $q$ maps a basis of $F^0$ bijectively to a minimal homogenous set of generators of $M$, and
$(2)$ $p$ maps a basis of $F^1$ bijectively to a minimal homogenous set of generators of $\mathrm{ker}(q)$.
This definition agrees with the definition of \cite{lw-computing}
by the analogue for persistence modules of \cite[Lemma 19.4]{eisenbud}.
Any presentation of a persistence module $M$ can be obtained (up to isomorphism) from a minimal presentation $p$
by taking the direct sum with maps of the form $\mathrm{id}_H : H \rightarrow H$
or of the form $H \rightarrow 0$, where $H$ is free.
Following Lesnick-Wright, we say that a presentation is \emph{semi-minimal}
if it can be obtained from a minimal presentation by taking the direct sum with maps of the form
$\mathrm{id}_H : H \rightarrow H$, where $H$ is free.

If $p : F^1 \rightarrow F^0$ is a presentation of $M$, and we choose ordered bases $B^0$ and $B^1$ of $F^0$ and $F^1$,
then we can represent $p$ by a bi-graded matrix $P$,
and we abuse notation by also calling this bi-graded matrix a presentation of $M$.
We write $P(i,j)$ for the entry of $P$ in position $(i,j)$,
and we write $P(*,j)$ for the $j^{th}$ column of $P$.

\begin{proposition}
  If $P$ is a semi-minimal presentation of a persistence module $M$,
  and $P'$ is the bi-graded matrix obtained from $P$ by the lazy minimization algorithm,
  then $P'$ is a minimal presentation of $M$.
\end{proposition}

\begin{proof}
  We first show that $P'$ is still a presentation of $M$.
  For this, we observe first that if $k$ and $j$ are column indices of $P$,
  with $k \neq j$ and such that the grade of column $k$ is less than or equal to the grade of column $j$
  (in the usual partial order on $\Z^2$),
  then adding column $k$ of $P$ to column $j$ corresponds to multiplying $P$ on the right with a bi-graded
  elementary matrix, and therefore does not change the isomorphism type of the cokernel.
  It follows that the column operations performed by the lazy minimization algorithm do not change
  the isomorphism type of the cokernel.

  Now, if $P$ has a row $q$ and a column $j$ that both have grade $\textbf{z} = (z_1, z_2)$, such that $P(q, j) = 1$,
  and every other entry in row $q$ and column $j$ is zero, then we can remove row $q$ and column $j$
  from $P$ without changing the isomorphism type of the cokernel;
  this corresponds to removing a direct summand of the form
  $\mathrm{id}_H : H \rightarrow H$, where $H$ is a free persistence module on one generator at grade $\textbf{z}$.
  Furthermore, the assumption that column $j$ has only one non-zero entry can be weakened.
  Assume now that $P$ has a column $j$ with pivot $q$,
  both with grade $\textbf{z}$,
  and such that $P(q,k) = 0$ for all $k \neq j$.
  Let $C = (q_1, \dots, q_r, q)$ be the row indices in which column $j$ is non-zero,
  and say that $P$ represents a homomorphism $p : F^1 \rightarrow F^0$,
  with respect to ordered bases $B^0 = (B^0_1, \dots, B^0_m)$ and
  $B^1 = (B^1_1, \dots, B^1_n)$.
  Define a basis $\tilde{B}^0 = (\tilde{B}^0_1, \dots, \tilde{B}^0_m)$ of $F^0$,
  where $\tilde{B}^0_t = B^0_t$ for $t \neq q$, and
  \[
  \tilde{B}^0_q = \sum_{s \in C} x^{z_1 - \gr(s)_1} y^{z_2 - \gr(s)_2} B^0_s \; .
  \]
  If $\tilde{P}$ represents $p$ with respect to the bases $\tilde{B}^0$ and $B^1$,
  then $\tilde{P}(*,k) = P(*,k)$ for all $k \neq j$,
  and $\tilde{P}(*,j)$ has a $1$ in row $q$ and zero elsewhere.
  So, by first replacing $B^0$ with $\tilde{B}^0$,
  we can remove row $q$ and column $j$ from $P$ without changing the isomorphism type of the cokernel.
  In the last step of the lazy minimization algorithm,
  in which all marked row-column pairs are removed,
  if we remove these row-column pairs in order of decreasing row index,
  then the previous argument shows that each row-column removal does not change the isomorphism type of the cokernel.
  This completes the proof that $P'$ is still a presentation of $M$.

  Finally, we show that $P'$ is minimal.
  Since the LW minimization algorithm produces a minimal presentation from $P$,
  it suffices to show that the lazy minimization algorithm removes the same number of row-column pairs
  at each grade as the LW minimization algorithm.
  So, let $\textbf{z}$ be a grade of a column of $P$, and let $P_{\textbf{z}}$ be the sub-matrix of $P$
  consisting of the columns and rows of $P$ with grade $\textbf{z}$.
  To complete the proof, we show that 
  the number of row-column pairs with grade $\textbf{z}$ removed by the LW minimization algorithm
  and the number of row-column pairs with grade $\textbf{z}$ removed by the lazy minimization algorithm
  are both equal to the rank of the matrix $P_{\textbf{z}}$.
  Say that a matrix $N$ has the ``distinct pivot'' property if any pair of non-zero columns of $N$ have
  distinct pivots. Note that if $N$ has the distinct pivot property, then the rank of $N$ is the number
  of non-zero columns of $N$.
  The column operations performed by the LW minimization algorithm do not change the rank of $P_{\textbf{z}}$,
  and after performing all column operations, the sub-matrix $P_{\textbf{z}}$
  has the distinct pivot property, and the columns with grade $\textbf{z}$
  removed by the LW minimization algorithm are the non-zero columns of this matrix.
  Similarly, the column operations performed by the lazy minimization algorithm do not change the rank of $P_{\textbf{z}}$,
  and after performing all column operations, the sub-matrix $P_{\textbf{z}}$
  has the distinct pivot property, and the columns with grade $\textbf{z}$
  removed by the lazy minimization algorithm are the non-zero columns of this matrix.
\end{proof}

\paragraph{Correctness of queuing.}
The correctness of the two techniques using priority queues follows more or less directly
from the description given in the main body. We still give a formal proof for completeness.
We start with the priority queue for the grades, assuming that each grade is handled
by scanning all columns from left to right as in the RIVET algorithm.
\begin{proposition}
Using a priority queue for the grades as described, the algorithms \texttt{min\_gens}
and \texttt{ker\_basis} produce the same output as in the LW algorithm.
\end{proposition}
\begin{proof}
The argument is mostly the same for \texttt{min\_gens} and \texttt{ker\_basis},
and we just talk about ``the algorithm'' for either of them. More precisely, we 
refer to the LW version of the algorithm and the optimized version of the algorithm
when talking about the variant without and with the priority queue, respectively.

We call a grade $(x,y)$ \emph{significant} if the algorithm, on grade $(x,y)$
performs a column operation on the matrix or appends a column to the output matrix.
We show that every significant grade in the LW algorithm is added to the priority
queue in the optimized version. That proves that the outcome is the same using an inductive argument.

Fix a grade $(x,y)$ and assume that the algorithm (in the LW version) performs
a column operation in this iteration. Let $i$ be the smallest index on which
such an operation is performed, let $j$ be its pivot, 
and let $(x',y)$ be its grade with $x'\leq x$.
If $x'=x$, then there is a matrix column with grade $(x,y)$ and the grade
is pushed into the grade priority queue initially. Otherwise, if $x'<x$,
column $i$ has been reduced previously in grade $(x-1,y)$, and the reduction
ended with $\rho(j)=i$ (otherwise, a further column addition would have been performed).
The fact that a column addition is needed at grade $(i,j)$ means that $\rho(j)$ must
have been re-set, to an index smaller than $i$. However, between $(x-1,y)$ and $(x,y)$,
the algorithms iterated over the grades $(x-1,y+1),(x-1,y+2),\ldots,(x-1,Y),(x,1),(x,2),\ldots (x,y-1)$
with $Y$ the maximal $y$-grade. Note that in the first part of the sequence, with $x$-grade $x-1$,
only columns with index $>i$ are updated because the matrix is stored in colex order.
Hence, none of these iterations can set $\rho(j)$ to a smaller index.
It follows that the update of $\rho(j)$ happens at a grade with $x$-grade $x$.
But then, the updated reduction algorithm ensures when $\rho(j)$ is updated,
the grade $(x,y)$ is added to the priority queue in this step.

It remains the case that at grade $(x,y)$, the algorithm appends an output element. 
If the algorithm is \texttt{min\_gens}, this only happens when the grade appears as matrix
column, and as argued earlier, these grades are added to the priority queue.
For \texttt{ker\_basis}, an output column is added if a column is reduced from
non-zero to zero, and this implies that at least one column addition was performed, so
the grade is considered by the first part.
\end{proof}

For the second queuing optimization (to avoid the left-to-right scan at each grade),
the proof strategy is very similar:
\begin{proposition}
Using one priority queue per $y$-grades as described, the algorithms \texttt{min\_gens}
and \texttt{ker\_basis} produce the same output as in the LW algorithm.
\end{proposition}
\begin{proof}
Again, we argue inductively that the same column operations are performed
in the variant without and with priority queues. Fix a column $i$
for which the LW variant (without priority queues) performs a column operation
for grade $(x,y)$. It suffices to show that $i$ is pushed to the priority queue
for $y$. If $i$ is at grade $(x,y)$, it is pushed by the algorithm as specified,
so we can assume that its grade is $(x',y)$ with $x'<x$.

Let $j$ be the pivot of column $i$. As in the previous proof,
the fact that a column addition is performed means that $\rho(j)$ has been updated
during the algorithm since $i$ was visited at grade $(x-1,y)$, and this can only
happen at a grade $(x,y')$ with $y'\leq y$. By the modified reduction method,
$i$ will be pushed into the priority queue of $y$ in this case.
\end{proof}

\section{Additional experimental data}
\label{app:experiments_appendix}

In this section, we present further details on the experiments discussed in
Section~\ref{sec:further_experiments}.

\paragraph{Function-Rips data.}

Along with the function-Rips bifiltration on a noisy circle point cloud,
which was discussed in Section~\ref{sec:further_experiments},
we also considered the function-Rips bifiltration on a noisy $2$-sphere in $\R^3$,
and created a persistence module by taking homology in dimension $2$.
The results are in Table~\ref{tbl:function_rips_2_sphere}.
As was the case with the noisy circle,
the runtime of the LW algorithm is dominated by \texttt{Min\_gens}.
And again, since some of these reductions can be performed during chunk preprocessing,
which can be done in parallel, we obtain a running time improvement of a factor of around $8$
in the largest example.

\paragraph{Mesh data from AIM@SHAPE.}

In Table~\ref{tbl:mesh_statistics} we display information about the size of the datasets we consider.
The numbers $\ell$,$n$,$m$ give the size of the input, as in Section~\ref{sec:definitions},
and $X$,$Y$ are the number of grades, as in Section~\ref{sec:lesnick-wright}.
Time and memory consumption are displayed in Table~\ref{tbl:mesh_runtime}.
The rivet clone runs out of memory for all but the smallest instance,
due to the large grid that needs to be created without the sparse grid optimization.
Similar to the examples in Section~\ref{sec:improvements},
we see a clear advantage of using chunk preprocessing, both with respect to time and memory.
On the other hand, neither parallelization nor the clearing optimization (not shown in the table) 
seem to improve the algorithm's performance for these instances.

\paragraph{Random Delaunay triangulations, $d=2$.}
It might be surprising at first that the resulting presentation for this case is empty for all instances
(Tables~\ref{tbl:random_del_1_small} and~\ref{tbl:random_del_1_large}).
To explain this, note that every generator in such a presentation corresponds
to a void in the data at some grade $p=(x,y)$ which is bounded by a set of triangles of grade $\leq p$.
Set $P:=(-\infty,x)\times (-\infty,y)\times \R$.
In other words, the void is a polyhedron whose vertices lie in $P$. Since $P$ is a convex set,
this implies that the void is contained in $P$ as well. On the other hand, the void is 
triangulated by tetrahedra which are not present at grade $p$, as otherwise there would not be a void. 
But this is a contradiction, since for any such tetrahedron, its vertices are all in $P$,
and hence so is the tetrahedron, by the way the grades are assigned. It follows that such a void
does not exist.

\begin{table*}[h]\centering
\begin{tabular}{c|l|cccccc|cc|c}
N & Variant & IO & Ch & MG & KB & RP & Min & Time & Mem & Size\\
\hline
\multirow{4}{*}{271K}
& rivet\_clone & 0.29 & - & 11.1 & 0.08 & 0.03 & 1.60 & 13.1 & 336MB & \multirow{4}{*}{(7.8, 4.2)}\\
& +queue,lazy & 0.31 & - & 10.1 & 0.06 & 0.03 & 0.00 & 10.5 & 336MB\\
& +chunk & 0.30 & 9.85 & 0.12 & 0.00 & 0.00 & 0.00 & 10.3 & 406MB\\
& +parfor & 0.31 & 2.32 & 0.11 & 0.00 & 0.00 & 0.00 & 2.78 & 424MB\\
\hline
\multirow{4}{*}{522K}
& rivet\_clone & 0.56 & - & 41.0 & 0.20 & 0.06 & 5.72 & 47.6 & 861MB & \multirow{4}{*}{(8.6, 5.6)}\\
& +queue,lazy & 0.56 & - & 37.0 & 0.15 & 0.06 & 0.01 & 37.8 & 863MB\\
& +chunk & 0.57 & 35.7 & 0.47 & 0.02 & 0.00 & 0.00 & 36.8 & 1.05GB\\
& +parfor & 0.60 & 7.32 & 0.35 & 0.00 & 0.00 & 0.00 & 8.31 & 1.11GB\\
\hline
\multirow{4}{*}{1.03M}
& rivet\_clone & 1.10 & - & 136 & 0.50 & 0.12 & 19.2 & 157 & 2.20GB & \multirow{4}{*}{(11.0, 6.8)}\\
& +queue,lazy & 1.11 & - & 122 & 0.38 & 0.12 & 0.02 & 123 & 2.20GB\\
& +chunk & 1.11 & 117 & 1.66 & 0.10 & 0.00 & 0.00 & 120 & 2.70GB\\
& +parfor & 1.14 & 21.9 & 1.05 & 0.00 & 0.00 & 0.00 & 24.2 & 2.91GB\\
\hline
\multirow{4}{*}{2.12M}
& rivet\_clone & 2.30 & - & 573 & 1.46 & 0.30 & 70.9 & 648 & 6.39GB & \multirow{4}{*}{(13.2, 9.0)}\\
& +queue,lazy & 2.32 & - & 508 & 1.18 & 0.30 & 0.08 & 512 & 6.40GB\\
& +chunk & 2.32 & 489 & 5.78 & 0.55 & 0.00 & 0.00 & 498 & 7.63GB\\
& +parfor & 2.38 & 84.7 & 2.74 & 0.00 & 0.00 & 0.00 & 90.0 & 8.15GB\\
\hline
\multirow{4}{*}{4.25M}
& rivet\_clone & 4.66 & - & 2090 & 4.41 & 0.75 & 299 & 2400 & 17.2GB & \multirow{4}{*}{(21.0, 13.0)}\\
& +queue,lazy & 4.67 & - & 1839 & 3.74 & 0.75 & 0.28 & 1849 & 17.2GB\\
& +chunk & 4.71 & 1761 & 23.0 & 0.68 & 0.00 & 0.00 & 1790 & 20.7GB\\
& +parfor & 4.85 & 294 & 10.5 & 0.01 & 0.02 & 0.00 & 310 & 22.4GB\\
\end{tabular}
\caption{Time and memory consumption for the function-Rips bifiltration on a noisy $2$-sphere.
  $N:=\ell+n$ is the total number of columns in $A$ and $B$.
      The meaning of the other columns is the same as in Table~\ref{tbl:chunk_et_al}.}
\label{tbl:function_rips_2_sphere}
\end{table*}

\begin{table}[h]\centering
\begin{tabular}{c|c|c|c}
name & $\ell$,$n$,$m$ & $X$,$Y$ & Size\\
\hline
hand   & 69K,  103K, 34K  & 33K,34K & 595$\times$596\\
eros   & 953K, 1.4M, 477K& 55K,58K & 3K$\times$3K \\
dragon & 1.3M,2.0M, 656K & 139K,189K & 8K$\times$8K\\
raptor & 2.0M,3.0M, 1.0M & 565K,523K & 7K$\times$7K
\end{tabular}
\caption{The statistics of the triangular mesh benchmark data.
  The numbers $\ell$,$n$,$m$ give the size of the input,
  $X$,$Y$ are the number of grades, and ``Size''
  is the size of the output.}
\label{tbl:mesh_statistics}
\end{table}

\begin{table*}[h]\centering
\begin{tabular}{c|l|cccccc|cc}
name & Variant & IO & Ch & MG & KB & RP & Min & Time & Mem\\
\hline
\multirow{3}{*}{hand}
& rivet\_clone & 0.37 & - & 34.2 & 37.8 & 0.07 & 12.0 & 84.8 & 8.79GB\\
& +queue,lazy & 0.37 & - & 0.08 & 0.32 & 0.07 & 0.03 & 0.91 & 147MB\\
& +chunk & 0.37 & 0.02 & 0.00 & 0.04 & 0.00 & 0.00 & 0.46 & 76MB\\
& +parfor & 0.44 & 0.03 & 0.00 & 0.04 & 0.00 & 0.00 & 0.63 & 76MB\\
%& +clearing & 0.38 & 0.03 & 0.00 & 0.04 & 0.00 & 0.00 & 0.48 & 75MB\\
\hline
\multirow{2}{*}{eros}
& queue,lazy & 3.15 & - & 1.64 & 11.8 & 1.41 & 0.81 & 19.0 & 2.94GB\\
& +chunk & 3.20 & 0.47 & 0.16 & 1.37 & 0.11 & 0.02 & 5.4 & 873MB\\
& +parfor & 3.50 & 0.43 & 0.16 & 1.36 & 0.05 & 0.02 & 5.64 & 873MB\\
%& +clearing & 3.18 & 0.43 & 0.17 & 1.35 & 0.06 & 0.02 & 5.29 & 872MB\\
\hline
\multirow{2}{*}{dragon}
& queue,lazy & 5.58 & - & 2.84 & 14.0 & 2.03 & 1.18 & 26.0 & 3.40GB\\
& +chunk & 5.34 & 0.78 & 0.34 & 2.21 & 0.20 & 0.06 & 9.05 & 1.17GB\\
& +parfor & 5.70 & 0.70 & 0.32 & 2.23 & 0.11 & 0.05 & 9.23 & 1.17GB\\
%& +clearing & 5.35 & 0.70 & 0.34 & 2.19 & 0.11 & 0.04 & 8.84 & 1.17GB\\
\hline
\multirow{2}{*}{raptor}
& queue,lazy & 10.1 & - & 4.21 & 20.7 & 3.11 & 2.24 & 41.2 & 5.15GB\\
& +chunk & 10.1 & 1.27 & 0.29 & 2.48 & 0.19 & 0.06 & 14.7 & 1.88GB\\
& +parfor & 11.1 & 1.16 & 0.29 & 2.68 & 0.11 & 0.04 & 15.6 & 1.87GB\\
%& +clearing & 10.1 & 1.09 & 0.30 & 2.46 & 0.10 & 0.04 & 14.3 & 1.87GB\\
\end{tabular}
    \caption{Time and memory consumption for the triangular mesh dataset.
      The meaning of the columns is the same as in Table~\ref{tbl:chunk_et_al}.}
\label{tbl:mesh_runtime}
\end{table*}

\begin{table*}[h]\centering
\begin{tabular}{c|l|cccccc|cc|c}
n & Variant & IO & Ch & MG & KB & RP & Min & Time & Mem & Size\\
\hline
\multirow{4}{*}{5000}
& rivet\_clone & 0.14 & - & 1.28 & 0.99 & 0.02 & 2.69 & 5.16 & 242MB & \multirow{4}{*}{(560.0, 282.4)}\\
& +queue,lazy & 0.14 & - & 0.09 & 0.06 & 0.03 & 0.01 & 0.35 & 55MB\\
& +chunk & 0.14 & 0.11 & 0.00 & 0.00 & 0.00 & 0.00 & 0.26 & 41MB\\
& +parfor & 0.15 & 0.05 & 0.00 & 0.00 & 0.00 & 0.00 & 0.21 & 39MB\\
\hline
\multirow{4}{*}{10000}
& rivet\_clone & 0.31 & - & 5.42 & 4.42 & 0.06 & 12.2 & 22.4 & 883MB & \multirow{4}{*}{(1178.0, 593.2)}\\
& +queue,lazy & 0.30 & - & 0.22 & 0.15 & 0.06 & 0.04 & 0.80 & 120MB\\
& +chunk & 0.30 & 0.33 & 0.01 & 0.00 & 0.00 & 0.00 & 0.67 & 80MB\\
& +parfor & 0.30 & 0.14 & 0.01 & 0.01 & 0.00 & 0.00 & 0.47 & 79MB\\
\hline
\multirow{4}{*}{20000}
& rivet\_clone & 0.63 & - & 27.9 & 19.0 & 0.15 & 57.9 & 105 & 3.35GB & \multirow{4}{*}{(2345.8, 1183.0)}\\
& +queue,lazy & 0.61 & - & 0.57 & 0.36 & 0.16 & 0.12 & 1.87 & 255MB\\
& +chunk & 0.62 & 1.04 & 0.04 & 0.02 & 0.00 & 0.00 & 1.75 & 158MB\\
& +parfor & 0.61 & 0.38 & 0.05 & 0.02 & 0.00 & 0.00 & 1.08 & 158MB\\
\hline
\multirow{4}{*}{40000}
& rivet\_clone & 1.54 & - & 150 & 93.7 & 0.36 & 415 & 661 & 13.0GB & \multirow{4}{*}{(4703.0, 2374.0)}\\
& +queue,lazy & 1.28 & - & 1.54 & 0.85 & 0.36 & 0.37 & 4.46 & 577MB\\
& +chunk & 1.28 & 3.55 & 0.13 & 0.06 & 0.02 & 0.00 & 5.08 & 317MB\\
& +parfor & 1.25 & 1.12 & 0.13 & 0.06 & 0.00 & 0.00 & 2.61 & 317MB\\
\end{tabular}
\caption{Random Delaunay triangulations with $d=1$, small instances.
The meaning of the columns is the same as in Table~\ref{tbl:chunk_et_al}.}
\label{tbl:random_del_1_small}
\end{table*}

\begin{table*}[h]\centering
\begin{tabular}{c|l|cccccc|cc|c}
n & Variant & IO & Ch & MG & KB & RP & Min & Time & Mem & Size\\
\hline
\multirow{3}{*}{80000}
& queue,lazy & 3.14 & - & 4.35 & 2.02 & 0.80 & 1.23 & 11.6 & 1.29GB & \multirow{3}{*}{(9438.0, 4762.2)}\\
& +chunk & 2.70 & 13.2 & 0.40 & 0.17 & 0.08 & 0.01 & 16.6 & 715MB\\
& +parfor & 2.60 & 3.43 & 0.43 & 0.16 & 0.02 & 0.00 & 6.72 & 733MB\\
\hline
\multirow{3}{*}{160000}
& queue,lazy & 6.58 & - & 12.2 & 4.69 & 1.95 & 4.77 & 30.4 & 2.80GB & \multirow{3}{*}{(19275.8, 9722.6)}\\
& +chunk & 5.67 & 53.3 & 1.32 & 0.43 & 0.30 & 0.04 & 61.2 & 1.78GB\\
& +parfor & 5.47 & 11.3 & 1.46 & 0.41 & 0.07 & 0.02 & 18.9 & 2.04GB\\
\hline
\multirow{3}{*}{320000}
& queue,lazy & 13.8 & - & 32.6 & 11.2 & 4.79 & 21.3 & 84.3 & 6.52GB & \multirow{3}{*}{(38659.6, 19505.0)}\\
& +chunk & 11.8 & 227 & 4.69 & 1.20 & 1.28 & 0.12 & 247 & 4.80GB\\
& +parfor & 11.5 & 41.6 & 5.05 & 1.04 & 0.23 & 0.06 & 59.7 & 6.17GB\\
\hline
\multirow{3}{*}{640000}
& queue,lazy & 29.4 & - & 85.0 & 26.0 & 11.2 & 94.6 & 247 & 14.7GB & \multirow{3}{*}{(77704.2, 39200.4)}\\
& +chunk & 24.8 & 1019 & 18.1 & 3.95 & 6.11 & 0.44 & 1073 & 14.8GB\\
& +parfor & 24.0 & 169 & 18.8 & 2.93 & 0.99 & 0.22 & 217 & 20.1GB\\
\end{tabular}
\caption{Random Delaunay triangulations with $d=1$, large instances.
The meaning of the columns is the same as in Table~\ref{tbl:chunk_et_al}.}
\label{tbl:random_del_1_large}
\end{table*}

\begin{table*}[h]\centering
\begin{tabular}{c|l|cccccc|cc|c}
n & Variant & IO & Ch & MG & KB & RP & Min & Time & Mem & Size\\
\hline
\multirow{4}{*}{5000}
& rivet\_clone & 0.14 & - & 0.89 & 1.90 & 0.03 & 2.83 & 5.81 & 314MB & \multirow{4}{*}{(0.0, 0.0)}\\
& +queue,lazy & 0.14 & - & 0.01 & 0.28 & 0.03 & 0.00 & 0.49 & 127MB\\
& +chunk & 0.14 & 0.01 & 0.00 & 0.05 & 0.00 & 0.00 & 0.21 & 42MB\\
& +parfor & 0.15 & 0.01 & 0.00 & 0.05 & 0.00 & 0.00 & 0.22 & 41MB\\
\hline
\multirow{4}{*}{10000}
& rivet\_clone & 0.30 & - & 3.88 & 8.42 & 0.06 & 12.2 & 24.9 & 1.07GB & \multirow{4}{*}{(0.0, 0.0)}\\
& +queue,lazy & 0.30 & - & 0.04 & 0.79 & 0.07 & 0.02 & 1.25 & 306MB\\
& +chunk & 0.29 & 0.02 & 0.00 & 0.14 & 0.00 & 0.00 & 0.46 & 82MB\\
& +parfor & 0.29 & 0.03 & 0.00 & 0.14 & 0.00 & 0.00 & 0.48 & 85MB\\
\hline
\multirow{4}{*}{20000}
& rivet\_clone & 0.62 & - & 19.3 & 49.8 & 0.16 & 57.3 & 127 & 3.89GB & \multirow{4}{*}{(0.0, 0.0)}\\
& +queue,lazy & 0.62 & - & 0.10 & 2.43 & 0.16 & 0.07 & 3.44 & 793MB\\
& +chunk & 0.62 & 0.07 & 0.00 & 0.37 & 0.00 & 0.00 & 1.08 & 184MB\\
& +parfor & 0.61 & 0.07 & 0.00 & 0.40 & 0.00 & 0.00 & 1.10 & 194MB\\
\hline
\multirow{4}{*}{40000}
& rivet\_clone & 1.54 & - & 109 & 260 & 0.40 & 391 & 764 & 14.5GB & \multirow{4}{*}{(0.0, 0.0)}\\
& +queue,lazy & 1.24 & - & 0.24 & 7.64 & 0.41 & 0.18 & 9.81 & 2.13GB\\
& +chunk & 1.25 & 0.15 & 0.00 & 1.11 & 0.00 & 0.00 & 2.56 & 453MB\\
& +parfor & 1.24 & 0.14 & 0.00 & 1.17 & 0.00 & 0.00 & 2.61 & 475MB\\
\end{tabular}
\caption{Random Delaunay triangulations with $d=2$, small instances.
The meaning of the columns is the same as in Table~\ref{tbl:chunk_et_al}.}
\label{tbl:random_del_2_small}
\end{table*}

\begin{table*}[h]\centering
\begin{tabular}{c|l|cccccc|cc|c}
n & Variant & IO & Ch & MG & KB & RP & Min & Time & Mem & Size\\
\hline
\multirow{3}{*}{80000}
& queue,lazy & 3.09 & - & 0.55 & 23.8 & 1.15 & 0.45 & 29.3 & 5.60GB & \multirow{3}{*}{(0.0, 0.0)}\\
& +chunk & 2.63 & 0.33 & 0.01 & 3.34 & 0.02 & 0.00 & 6.40 & 1.12GB\\
& +parfor & 2.53 & 0.31 & 0.01 & 3.48 & 0.01 & 0.00 & 6.42 & 1.17GB\\
\hline
\multirow{3}{*}{160000}
& queue,lazy & 6.40 & - & 1.24 & 79.2 & 3.52 & 1.45 & 92.4 & 14.8GB & \multirow{3}{*}{(0.0, 0.0)}\\
& +chunk & 5.45 & 0.72 & 0.03 & 10.3 & 0.05 & 0.00 & 16.7 & 2.73GB\\
& +parfor & 5.26 & 0.67 & 0.03 & 10.5 & 0.03 & 0.00 & 16.6 & 2.83GB\\
\hline
\multirow{3}{*}{320000}
& queue,lazy & 13.2 & - & 2.77 & 285 & 13.9 & 4.68 & 321 & 41.5GB & \multirow{3}{*}{(0.0, 0.0)}\\
& +chunk & 11.3 & 1.57 & 0.08 & 35.9 & 0.13 & 0.01 & 49.4 & 7.41GB\\
& +parfor & 10.9 & 1.44 & 0.07 & 36.3 & 0.07 & 0.01 & 49.1 & 7.60GB\\
\hline
\multirow{3}{*}{640000}
& queue,lazy & 28.0 & - & 6.28 & 1670 & 215 & 17.5 & 1941 & 64.9GB & \multirow{3}{*}{(0.0, 0.0)}\\
& +chunk & 27.6 & 3.37 & 0.17 & 112 & 0.30 & 0.03 & 144 & 19.6GB\\
& +parfor & 22.5 & 3.03 & 0.15 & 113 & 0.16 & 0.02 & 140 & 19.9GB\\
\end{tabular}
\caption{Random Delaunay triangulations with $d=2$, large instances.
The meaning of the columns is the same as in Table~\ref{tbl:chunk_et_al}.}
\label{tbl:random_del_2_large}
\end{table*}

\begin{table*}[h]\centering
\begin{tabular}{c|l|cccccc|cc|c}
N & Variant & IO & Ch & MG & KB & RP & Min & Time & Mem & Size\\
\hline
\multirow{4}{*}{100K}
& rivet\_clone & 0.12 & - & 0.62 & 0.54 & 0.04 & 4.96 & 6.30 & 84MB & \multirow{5}{*}{(6066.2, 3068.6)}\\
& +queue,lazy & 0.13 & - & 0.05 & 0.08 & 0.04 & 0.02 & 0.34 & 86MB\\
& +chunk & 0.12 & 0.02 & 0.01 & 0.01 & 0.01 & 0.00 & 0.21 & 39MB\\
& +parfor & 0.13 & 0.02 & 0.01 & 0.02 & 0.00 & 0.00 & 0.20 & 38MB\\
%& +clearing & 0.12 & 0.02 & 0.01 & 0.01 & 0.00 & 0.00 & 0.19 & 38MB\\
\hline
\multirow{4}{*}{199K}
& rivet\_clone & 0.25 & - & 2.61 & 2.27 & 0.11 & 21.7 & 27.0 & 207MB & \multirow{5}{*}{(13029.8, 6588.8)}\\
& +queue,lazy & 0.26 & - & 0.10 & 0.21 & 0.11 & 0.06 & 0.77 & 210MB\\
& +chunk & 0.25 & 0.05 & 0.04 & 0.05 & 0.03 & 0.01 & 0.46 & 75MB\\
& +parfor & 0.25 & 0.04 & 0.04 & 0.05 & 0.01 & 0.01 & 0.44 & 87MB\\
%& +clearing & 0.27 & 0.04 & 0.04 & 0.04 & 0.02 & 0.00 & 0.44 & 78MB\\
\hline
\multirow{4}{*}{400K}
& rivet\_clone & 0.53 & - & 14.7 & 9.57 & 0.26 & 111 & 136 & 536MB & \multirow{5}{*}{(27425.0, 13858.4)}\\
& +queue,lazy & 0.57 & - & 0.24 & 0.57 & 0.26 & 0.20 & 1.90 & 545MB\\
& +chunk & 0.53 & 0.13 & 0.09 & 0.14 & 0.08 & 0.05 & 1.06 & 174MB\\
& +parfor & 0.53 & 0.10 & 0.10 & 0.15 & 0.04 & 0.02 & 0.98 & 214MB\\
%& +clearing & 0.56 & 0.10 & 0.10 & 0.12 & 0.05 & 0.02 & 0.98 & 186MB\\
\hline
\multirow{4}{*}{801K}
& rivet\_clone & 1.13 & - & 91.8 & 50.9 & 0.65 & 697 & 842 & 1.31GB & \multirow{5}{*}{(57810.4, 29202.0)}\\
& +queue,lazy & 1.20 & - & 0.54 & 1.46 & 0.65 & 0.57 & 4.52 & 1.32GB\\
& +chunk & 1.12 & 0.29 & 0.21 & 0.40 & 0.22 & 0.15 & 2.45 & 420MB\\
& +parfor & 1.11 & 0.21 & 0.22 & 0.42 & 0.12 & 0.08 & 2.19 & 528MB\\
%& +clearing & 1.13 & 0.21 & 0.24 & 0.34 & 0.12 & 0.06 & 2.15 & 450MB\\
\end{tabular}
\caption{Multi-cover data, small instances.
$N:=\ell+n$ is the total number of columns in $A$ and $B$.
      The meaning of the other columns is the same as in Table~\ref{tbl:chunk_et_al}.}
\label{tbl:multicover_small}
\end{table*}

\begin{table*}[h]\centering
\begin{tabular}{c|l|cccccc|cc|c}
N & Variant & IO & Ch & MG & KB & RP & Min & Time & Mem & Size\\
\hline
\multirow{3}{*}{1.6M}
& queue,lazy & 3.06 & - & 1.23 & 4.18 & 1.69 & 1.90 & 12.2 & 3.67GB & \multirow{4}{*}{(119379.6, 60273.4)}\\
& +chunk & 2.38 & 0.60 & 0.49 & 1.21 & 0.61 & 0.50 & 5.86 & 1.15GB\\
& +parfor & 2.32 & 0.40 & 0.51 & 1.25 & 0.32 & 0.21 & 5.1 & 1.49GB\\
%& +clearing & 2.31 & 0.41 & 0.56 & 1.00 & 0.31 & 0.16 & 4.83 & 1.23GB\\
\hline
\multirow{3}{*}{3.17M}
& queue,lazy & 7.26 & - & 2.81 & 11.6 & 4.61 & 6.20 & 32.9 & 9.50GB & \multirow{4}{*}{(244096.4, 123224.4)}\\
& +chunk & 5.09 & 1.28 & 1.14 & 3.47 & 1.65 & 1.55 & 14.3 & 3.04GB\\
& +parfor & 4.98 & 0.86 & 1.19 & 3.50 & 0.74 & 0.54 & 11.9 & 4.07GB\\
%& +clearing & 4.95 & 0.86 & 1.29 & 2.76 & 0.73 & 0.40 & 11.1 & 3.24GB\\
\hline
\multirow{3}{*}{6.49M}
& queue,lazy & 20.2 & - & 6.44 & 32.8 & 14.5 & 22.3 & 97.3 & 27.1GB & \multirow{4}{*}{(509904.2, 257387.4)}\\
& +chunk & 11.4 & 2.80 & 2.69 & 10.1 & 4.72 & 5.10 & 37.2 & 8.78GB\\
& +parfor & 11.0 & 1.83 & 2.75 & 10.2 & 1.82 & 1.38 & 29.4 & 12.3GB\\
%& +clearing & 10.8 & 1.83 & 3.03 & 7.89 & 2.09 & 1.05 & 27.1 & 9.41GB\\
\end{tabular}
\caption{Multi-cover data, large instances.
$N:=\ell+n$ is the total number of columns in $A$ and $B$.
      The meaning of the other columns is the same as in Table~\ref{tbl:chunk_et_al}.}
\label{tbl:multicover_large}
\end{table*}

\begin{table*}[h]\centering
\begin{tabular}{c|l|cccccccc}
name, N & Variant & Vector & Heap & List & Set & A-Full & A-Set & A-Heap & A-Bit-Tree\\
\hline
multi-cover & lazy, $\ldots$, chunk & \textbf{6.56} & 24.5 & 17.8 & 23.7 & 12.0 & 21.0 & 22.3 & 7.43\\
1.59M & +parfor & \textbf{5.42} & 18.5 & 16.2 & 19.2 & 9.04 & 15.8 & 15.8 & 5.81\\
\hline
multi-cover & lazy, $\ldots$, chunk & \textbf{37.8} & 189 & 166 & 180 & 82.0 & 185 & 169 & 45.9\\
6.49M & +parfor & \textbf{28.6} & 138 & 118 & 141 & 60.9 & 135 & 118 & 35.6\\
\hline
func. Rips & lazy, $\ldots$, chunk & \textbf{6.73} & 22.5 & 29.4 & 16.9 & 11.0 & 16.1 & 24.0 & 8.33\\
1.02M & +parfor & 4.64 & 11.4 & 21.5 & 9.71 & 5.88 & 7.86 & 10.5 & \textbf{4.03}\\
\hline
func. Rips & lazy, $\ldots$, chunk & \textbf{108} & 435 & 698 & 288 & 180 & 277 & 459 & 125\\
8.17M & +parfor & 65.7 & 211 & 539 & 162 & 89.0 & 131 & 198 & \textbf{50.6}\\
\hline
hand & lazy, $\ldots$, chunk & \textbf{0.52} & 0.53 & 0.73 & 0.56 & 0.54 & 0.59 & 0.54 & 0.53\\
172K & +parfor & \textbf{0.49} & 0.53 & 0.78 & 0.57 & 0.5 & 0.52 & 0.52 & \textbf{0.49}\\
\hline
dragon & lazy, $\ldots$, chunk & \textbf{9.04} & 11.4 & 40.4 & 11.4 & 10.2 & 11.5 & 11.1 & 9.51\\
3.28M & +parfor & \textbf{8.73} & 10.5 & 53.7 & 10.9 & 9.19 & 10.5 & 9.91 & 8.85\\
\hline
random Del. ($d=1$) & lazy, $\ldots$, chunk & \textbf{16.9} & 50.8 & 104 & 34.5 & 22.6 & 32.5 & 51.0 & 17.1\\
1.68M & +parfor & 6.73 & 14.0 & 29.5 & 12.3 & 8.25 & 11.6 & 13.1 & \textbf{6.37}\\
\hline
random Del. ($d=1$) & lazy, $\ldots$, chunk & 248 & 697 & 2381 & 451 & 295 & 431 & 687 & \textbf{213}\\
6.75M & +parfor & 59.5 & 165 & 630 & 138 & 79.6 & 150 & 152 & \textbf{53.8}\\
\hline
random Del. ($d=2$) & lazy, $\ldots$, chunk & 6.68 & 23.9 & 48.7 & 23.8 & 9.67 & 24.4 & 19.8 & \textbf{6.35}\\
1.60M & +parfor & 6.28 & 23.6 & 72.2 & 24.4 & 9.26 & 23.5 & 19.2 & \textbf{6.14}\\
\hline
random Del. ($d=2$) & lazy, $\ldots$, chunk & 52.1 & 254 & 1196 & 262 & 70.1 & 312 & 196 & \textbf{38.9}\\
6.43M & +parfor & 49.9 & 251 & 1620 & 266 & 68.1 & 274 & 193 & \textbf{38.0}\\
\hline
convex hull & lazy, $\ldots$, chunk & \textbf{7.6} & 9.31 & 17.9 & 9.48 & 8.33 & 9.33 & 9.17 & 7.74\\
2.00M & +parfor & \textbf{7.18} & 8.78 & 17.2 & 9.46 & 7.76 & 8.44 & 8.38 & 7.59\\
\hline
convex hull & lazy, $\ldots$, chunk & \textbf{32.3} & 42.3 & 158 & 42.5 & 36.2 & 45.3 & 41.4 & 32.5\\
8.00M & +parfor & \textbf{29.6} & 39.0 & 119 & 40.8 & 33.8 & 38.9 & 39.0 & 30.7\\
\end{tabular}
    \caption{Running times for the different column types in \textsc{Phat}. 
      From left to right: $N:=\ell+n$ is the total number of columns in $A$ and $B$,
      \texttt{vector\_vector}, \texttt{vector\_heap}, \texttt{vector\_list},
      \texttt{vector\_set}, \texttt{full\_pivot\_column},  \texttt{sparse\_pivot\_column},
      \texttt{heap\_pivot\_column}, \texttt{bit\_tree\_pivot\_column}.}
    \label{tbl:column_compare}
\end{table*}

\clearpage

\section{Clearing}
\label{app:clearing}

The clearing optimization (also hinted at in~\cite{lw-computing})
is based on the following idea:
When the algorithm \texttt{Min\_gens} returns a column, by definition, the column
encodes an element of the kernel of $B$. We would like to
simply use this element as a basis vector instead of computing a different
one through additional column operations in \texttt{Ker\_basis}.
Indeed, if the column has grade $p$, then it represents
a kernel element at grade $p$.
The problem, however, is that this kernel element might be present at grades
smaller than $p$ as well, and it seems difficult
to determine the ``birth grade'' of this kernel element in general,
and which kernel element in \texttt{ker\_basis} can be replaced
with that column.

There are cases, however, where this question can be resolved:
call a non-zero column \emph{pivot-dominated}, 
if each non-zero row entry of the column is at a grade that is smaller than the grade of the pivot.
In particular, local columns are pivot-dominated. 
We then proceed as follows: At the beginning
of \texttt{Ker\_basis}, iterate over the columns
returned by \texttt{Min\_gens}. If a column is pivot-dominated
with pivot $i$, we add this column to the kernel immediately, with the same
grade as the row grade of $i$ in $A$, and we set column $i$ in $B$ to zero
(it is possible that the same column $i$ is set by
several columns of \texttt{Min\_gens}~-- this does not invalidate the algorithm,
since either of them yields a valid outcome).

\begin{lemma}
The modified algorithm yields a valid kernel basis.
\end{lemma}
\begin{proof}
Consider a pivot-dominated column $c$ returned by \texttt{Min\_gens}, with
pivot index $i$ at grade $p$. First, we argue that the unmodified
\texttt{ker\_basis} procedure will return a kernel element when visiting column $i$
at grade $p$: indeed, this happens if and only if there is a linear combination in $B$
that involves column $i$ and other columns with grade $\leq p$. 
But since $c$ is pivot-dominated,
it encodes precisely such a linear combination.
So, letting $c'$ denote the column produced by the unmodified algorithm at column $i$,
we have that $c$ and $c'$ differ by a linear combination of elements with index $<i$.
Hence, exchanging $c'$ with $c$ yields a basis for the kernel of $B_{\leq p}$,
and for all grades $\neq p$, the basis is unaffected. The statement follows
by applying this argument iteratively.
\end{proof}

The lemma implies that the resulting minimal presentation is still isomorphic
to the original one, proving the correctness of the algorithm. 
%We remark, however,
%that unlike all other modifications, the clearing optimization changes the
%minimal presentation matrix in general.

\paragraph{Practical evaluation.}
For most cases, the clearing optimization shows no practical benefit. 
There are several good reasons for this:

\begin{itemize}
\item The modification can only improve the performance of \texttt{ker\_basis}
(and potentially \texttt{reparam}) and hence is useless when 
\texttt{chunk} and/or \texttt{min\_gens}
are the computational bottleneck, as for instance the case in the function-Rips
datasets (Table~\ref{tbl:function_rips_runtime}).
\item When \texttt{Min\_gens} returns much fewer columns than \texttt{ker\_basis},
clearing will only save operations for a small fraction of columns. This happens
in particular when $A$ is much smaller than $B$.
\item To take any effect, there must be many pivot-dominated columns. 
However, using chunk preprocessing, there will not be any local column
returned by \texttt{Min\_gens} which is the most common example
of pivot-domination.
\end{itemize}

We observed that the fraction of pivot-dominated columns varies a lot between the examples.
For instance, in our convex hull examples, as well as for the meshes,
the number of such columns was very small (less than 5\% of all columns from \texttt{Min\_gens}).
For random Delaunay complexes and homology dimension $2$, there was not a single pivot-dominated
columns throughout, whereas for dimension $1$, more than $90\%$ of the columns
were pivot-dominated. Still, clearing has no significant effect in these examples either, because
the majority of running time is spent on chunk preprocessing.

\begin{table}[h]\centering
\begin{tabular}{c|c|c|c|c}

N & KB & KB$^\ast$ & $\Delta$ & MG\\
\hline
800K & 0.42 & 0.34 & 0.08 & 0.24\\
1.6M & 1.25 & 1.00 & 0.25 & 0.56\\
3.2M & 3.50 & 2.76 & 0.74 & 1.29\\
6.5M & 10.2 & 7.89 & 2.31 & 3.03
\end{tabular}
    \caption{The effect of the clearing optimization on the multi cover dataset.
From left to right: $N:=\ell+n$ is the total
number of columns in $A$ and $B$, time for \texttt{Ker\_basis} without clearing,
time for \texttt{Ker\_basis} with clearing, absolute difference of running times, 
time for \texttt{Min\_gens}.}
\label{tbl:multi_cover_clearing}
\end{table}

The only type of example where we could observe an effect on practical performance was
the multi-cover dataset: here, we observed that consistently
about half of the columns returned by \texttt{Min\_gens}
are pivot-dominated.
In Table~\ref{tbl:multi_cover_clearing}, we see that clearing
indeed yields speed-ups to \texttt{ker\_basis}, due to a decrease in column additions.
%Also the memory usage drops moderately.

Note that the clearing improvement requires that \texttt{Min\_gens} terminates before \texttt{ker\_basis}
starts. Therefore, it cannot be combined with the previously mentioned option 
of running \texttt{Min\_gens} and \texttt{ker\_basis} in parallel.
In this example, this parallel option leads to slightly better results,
as the last two columns in Table~\ref{tbl:multi_cover_clearing} show.

To summarize, many things have to come together to make clearing useful in practice;
we do not rule out the possibility that instances exists where it leads
to significant speedups, but such instances have to be identified in future work.

\section{Pseudocode}
\label{app:pseudocode}

In this section, we give pseudocode for all the algorithms discussed in the paper. Algorithms~\ref{alg:reduce_lw}--\ref{alg:lw_algo}
describe the LW algorithm, and Algorithms~\ref{alg:reduce_new}--\ref{alg:algorithm_new} our improved algorithm. More precisely,
we describe the version here that uses chunk preprocessing, queues, and lazy minimization. We also marked the for loops in the pseudo-code
that can be run in parallel.

\begin{algorithm*}
\caption{Matrix reduction, LW-version}
\label{alg:reduce_lw}
\begin{algorithmic}
\Function{reduce\_LW}{$A,i,use\_auxiliary=false$}
    \LeftComment{$A$ maintains a pivot vector $piv$}
    \LeftComment{$piv[i]=k$ means that column $k$ in $A$ has pivot $i$.}
    \LeftComment{$piv[i]=-1$ means that no (visited) column in $A$ has pivot $i$.}
    \LeftComment{If \textit{use\_auxiliary} is set, $A$ also maintains an auxiliary matrix.}
    \While{column $i$ in $A$ is not empty}
    \State {$j\gets$ pivot of column $i$}
    %\If {\texttt{local} and the row-grade of $j$ differs from the column-grade of $i$ in $A$}
    %\State {\textbf{break}}
    %\EndIf
    \State {$k\gets piv[j]$}
    \If {$k=-1$ or $k>j$}
    \State $piv[j]\gets i$
    \State {\textbf{break}}
    \Else
    \State {add column $k$ to column $i$}
    \If {\textit{use\_auxiliary}}
    \State {add auxiliary-column $k$ to auxiliary-column $i$}
    \EndIf
    \EndIf
    \EndWhile
\EndFunction
\end{algorithmic}
\end{algorithm*}

\begin{algorithm*}
\caption{Min\_gens, LW-version}
\label{alg:min_gens_lw}
\begin{algorithmic}
\Function{Min\_gens\_LW}{$A$}
    \State {$(X,Y) \gets grid\_dim\_of(A)$}\Comment{grades of $A$ are on $X\times Y$ grid}
    \State {$Out\gets\emptyset$}
    \For {$x=1,\ldots,X$}
    \For {$y=1,\ldots,Y$}
    \State {$L\gets$column indices of $A$ with grades $(0,y),\ldots(x-1,y)$ in order}
    \For {$i\in L$}
    \State {\Call{Reduce\_LW}{A,i}}
    \EndFor
    \State {$I\gets$ column indices of $A$ with grade $(x,y)$}
    \For {$i\in I$}
    \State {\Call{Reduce\_LW}{A,i}}
    \If {column $i$ is not zero}
    \State {Append column $i$ to $Out$ with grade $(x,y)$}
    \EndIf
    \EndFor
    \EndFor
    \EndFor
    \State{\Return{Out}}
\EndFunction
\end{algorithmic}
\end{algorithm*}

\begin{algorithm*}
\caption{Ker\_basis, LW-version}
\label{alg:ker_basis_lw}
\begin{algorithmic}
\Function{Ker\_basis\_LW}{$B$}
    \State {$(X,Y) \gets grid\_dim\_of(B)$}\Comment{grades of $B$ are on $X\times Y$ grid}
    \State {$Out\gets\emptyset$}
    \For {$x=1,\ldots,X$}
    \For {$y=1,\ldots,Y$}
    \State {$L\gets$column indices of $A$ with grades $(0,y),\ldots(x,y)$ in order}
    \For {$i\in L$}
    \State {\Call{Reduce\_LW}{A,i,use\_auxiliary=true}}
    \If {column $i$ has turned from non-zero to zero}
    \State {Append auxiliary-column of $i$ to $Out$ with grade $(x,y)$}
    \EndIf
    \EndFor
    \EndFor
    \EndFor
    \State{\Return {$Out$}}
\EndFunction
\end{algorithmic}
\end{algorithm*}

\begin{algorithm*}
\caption{Reparameterize}
\label{alg:reparam}
\begin{algorithmic}
\Function{Reparam}{$G,K$}
    \LeftComment{$G$ is the output of $\Call{Min\_gens}{A}$}
    \LeftComment{$K$ is the output of $\Call{Ker\_basis}{B}$}
    \State {Form matrix $(K|G)$ by concatenation.}
    \State {$Out\gets\emptyset$}
    \For {$i$ in index range of $G$-column in $(K|G)$} \Comment{Parallelizable}
    \State {\Call{Reduce\_LW}{$(K|G),i,use\_auxiliary=true$}}\Comment{Column $i$ is zero afterwards}
    \State {Append the auxiliary column $i$ to $Out$}
    \EndFor
    \State{\Return {$Out$}}
\EndFunction
\end{algorithmic}
\end{algorithm*}

\begin{algorithm*}
\caption{Locality test}
\begin{algorithmic}
\label{alg:locality}
\Function{is\_local}{$A,i$}
    \If {column $i$ in $A$ is $0$}
      \State{\Return {$true$}}
    \EndIf
    \State {$j\gets$ pivot of column $i$ in $A$}
    \State {Return true iff the column grade of $i$ in $A$ equals the row grade of $j$ in $A$}
\EndFunction
\end{algorithmic}
\end{algorithm*}

\begin{algorithm*}
\caption{Minimization, LW-version}
\label{alg:minimize_lw}
\begin{algorithmic}
\Function{Minimize\_LW}{$M'$}
\State {$n\gets\text{\# columns in $M'$}$} 
\For {$i=1,\ldots,n$} 
\If{\Call{is\_local}{$M',i$}}
\State {$j\gets$ pivot of column $i$ in $M'$}
\For{$k=i+1,\ldots,n$}
\If{column $k$ in $M'$ contains $j$ as row index}
\State{Add column $i$ to column $k$}
\EndIf
\EndFor
\State {Mark column $i$ and row $j$ in $M'$}
\EndIf
\EndFor
\State {$M\gets$ submatrix of $M'$ consisting of unmarked rows and columns}
\State {Re-index the columns of $M$ and \Return{$M$}}
\EndFunction
\end{algorithmic}
\end{algorithm*}

\begin{algorithm*}
\caption{The LW-algorithm}
\label{alg:lw_algo}
\begin{algorithmic}
\Function{Min\_pres\_LW}{$(A,B)$}
\State {$G\gets$\Call{Min\_gens\_LW}{$A$}}
\State {$K\gets$\Call{Ker\_basis\_LW}{$B$}}
\State {$M'\gets$\Call{Reparam}{$G,K$}}
\State {$M\gets$\Call{Minimize\_LW}{$M'$}}
\State {\Return $M$}
\EndFunction
\end{algorithmic}
\end{algorithm*}

%%%% Our version

\begin{algorithm*}
\caption{Matrix reduction, new version}
\label{alg:reduce_new}
\begin{algorithmic}
\Function{Reduce\_New}{$A,i,grade,use\_auxiliary=false,local\_check=false$}
    \LeftComment{$grade$ is the grade that the algorithm is currently considering}
    \LeftComment{On top of pivot vector and auxiliary vector, $A$ maintains priority queues:}
    \LeftComment{$cols\_at\_y\_grade[y_0]$ is a priority queue for column indices at $y$-grade $y_0$}
    \LeftComment{$grade\_queue$ is a priority queue of grades that the algorithm needs to look at}
    \While{column $i$ in $A$ is not empty}
    \State {$j\gets$ pivot of column $i$}
    \State {$k\gets piv[j]$}
    \If {$k=-1$}
    \State $piv[j]\gets i$
    \State {\textbf{break}}
    \EndIf
    \If {$k>j$}
    \State {$y\gets$ $y$-grade of column $k$ in $A$}
    \State {Push $k$ to $cols\_at\_y\_grade[y]$}
    \State {$x\gets$ $x$-grade of $grade$}
    \State {Push $(x,y)$ to $grade\_queue$}
    \State $piv[j]\gets i$
    \State {\textbf{break}}
    \EndIf
    \If {\texttt{local\_check} and not \Call{is\_local}{A,i}}
    \State {\textbf{break}}
    \EndIf
    \State {add column $k$ to column $i$}
    \If {\textit{use\_auxiliary}}
    \State {add auxiliary-column $k$ to auxiliary-column $i$}
    \EndIf
    \EndWhile
\EndFunction
\end{algorithmic}
\end{algorithm*}

\begin{algorithm*}
\caption{Min\_gens, new version}
\label{alg:min_gens_new}
\begin{algorithmic}
\Function{Min\_gens\_New}{$A$}
    \State {Initialize $grade\_queue$ as empty priority queue}
    \For {each column $c$ of $A$}
    \State{Push the grade of $c$ into $grade\_queue$}
    \EndFor
    \For {each $y$-grade $y_0$ that appears in a column of $A$}
    \State {Initialize $cols\_at\_y\_grade[y_0]$ as empty priority queue}
    \EndFor
    \State {$Out\gets\emptyset$}
    \While{$grade\_queue$ is not empty}
    \State{Pop the grade $(x,y)$ from $grade\_queue$ minimal in lex order}
    \State{Push all column indices at grade $(x,y)$ into $cols\_at\_y\_grade[y]$}
    \While{$cols\_at\_y\_grade[y]$ is not empty}
    \State{Pop column index $i$ from $cols\_at\_y\_grade[y]$ minimal in lex order}
    \State{\Call{Reduce\_new}{A,i,(x,y)}}\Comment{Might change the priority queues as side effect}
    \If{column $i$ is not $0$ and its grade is $(x,y)$}
    \State {Append column $i$ to $Out$ with grade $(x,y)$}
    \EndIf
    \EndWhile
    \EndWhile
    \State{\Return{Out}}
\EndFunction
\end{algorithmic}
\end{algorithm*}

\begin{algorithm*}
\caption{Ker\_basis, new version}
\label{alg:ker_basis_new}
\begin{algorithmic}
\Function{Ker\_basis\_new}{$B$}
    \State {Initialize $grade\_queue$ as empty priority queue}
    \For {each column $c$ of $B$}
    \State{Push the grade of $c$ into $grade\_queue$}
    \EndFor
    \For {each $y$-grade $y_0$ that appears in a column of $B$}
    \State {Initialize $cols\_at\_y\_grade[y_0]$ as empty priority queue}
    \EndFor
    \State {$Out\gets\emptyset$}
    \While{$grade\_queue$ is not empty}
    \State{Pop the grade $(x,y)$ from $grade\_queue$ minimal in lex order}
    \State{Push all column indices at grade $(x,y)$ into $cols\_at\_y\_grade[y]$}
    \While{$cols\_at\_y\_grade[y]$ is not empty}
    \State{Pop column index $i$ from $cols\_at\_y\_grade[y]$ minimal in lex order}
    \State{\Call{Reduce\_new}{A,i,(x,y),use\_auxiliary=true}}\Comment{Might change the priority queues as side effect}
    \If{column $i$ has turned from non-zero to zero}
    \State {Append auxiliary-column $i$ of $A$ to $Out$ with grade $(x,y)$}
    \EndIf
    \EndWhile
    \EndWhile
    \State{\Return{Out}}
\EndFunction
\end{algorithmic}
\end{algorithm*}

\begin{algorithm*}
\caption{Minimization, new version}
\label{alg:minimize_new}
\begin{algorithmic}
\Function{Minimize\_new}{$M'$}
\State {$n\gets\text{\# columns in $M'$}$} 
\For {$i=1,\ldots,n$} 
\State {$(x,y)\gets$ grade of column $i$}
\State {\Call{Reduce\_new}{M',i,(x,y),local\_check=true}}
\If {\Call{is\_local}{M',i}}
\State {$j \gets$ pivot of column $i$}
\State {Mark row $j$ and column $i$}
\EndIf
\EndFor
\For {each unmarked column $c$}\Comment{Parallelizable}
\State {$new\_col\gets$ empty list}
\While {$c$ is not $0$}
\State {Get maximal index $j$ from $c$}
\If {$j$ is marked} 
\State {$k\gets piv[j]$}
\State {Add column $k$ to $c$}
\Else
\State {Remove index $j$ from $c$ and append it to $new\_col$}
\EndIf
\EndWhile
\State{Set column $i$ of $M'$ to $new\_col$}
\EndFor
\State {$M\gets$ submatrix of $M'$ consisting of unmarked rows and columns}
\State {Re-index the columns of $M$ and \Return{$M$}}
\EndFunction
\end{algorithmic}
\end{algorithm*}

\begin{algorithm*}
\caption{Chunk preprocessing}
\label{alg:chunk}
\begin{algorithmic}
\Function{chunk}{$(A,B)$}
\State {$n\gets\text{\# columns in $A$}$} 
\For {$i=1,\ldots,n$} 
\State {$(x,y)\gets$ grade of column $i$ in $A$}
\State {\Call{Reduce\_new}{A,i,(x,y),local\_check=true}}
\If {\Call{is\_local}{A,i}}
\State {$j \gets$ pivot of column $i$ in $A$}
\State {Mark row $j$ and column $i$ of $A$ and mark column $j$ of $B$}
\EndIf
\EndFor
\For {each unmarked column $c$ in $A$}\Comment{Parallelizable}
\State {$new\_col\gets$ empty list}
\While {$c$ is not $0$}
\State {Get maximal index $j$ from $c$ and remove it from $c$}
\If {$j$ is marked} 
\State {$k\gets piv[j]$}
\State {Add column $k$ to $c$}
\Else
\State {Append index $j$ to $new\_col$}
\EndIf
\EndWhile
\State{Set column $i$ of $A$ to $new\_col$}
\EndFor
\State {$(A',B')\gets$ submatrices of $(A,B)$ consisting of unmarked rows and columns}
\State {Re-index the columns of $A'$ and $B'$ and \Return{$(A',B')$}}
\EndFunction
\end{algorithmic}
\end{algorithm*}

\begin{algorithm*}
\caption{The fast minimal presentation algorithm}
\label{alg:algorithm_new}
\begin{algorithmic}
\Function{Min\_pres\_new}{$(A,B)$}
\State {$(A',B')\gets$\Call{chunk}{$A$,$B$}}
\State {$G\gets$\Call{Min\_gens\_new}{$A'$}}
\State {$K\gets$\Call{Ker\_basis\_new}{$B'$}}
\State {$M'\gets$\Call{Reparam}{$G,K$}}
\State {$M\gets$\Call{Minimize\_new}{$M'$}}
\State {\Return $M$}
\EndFunction
\end{algorithmic}
\end{algorithm*}

\end{appendix}

\end{document}